\documentclass[11pt,reqno]{amsart}
\usepackage{amssymb}
\usepackage{amsfonts}
\usepackage{mathrsfs}
\usepackage{amsmath}
\usepackage{graphicx}
\usepackage{hyperref}
\usepackage{float}
\usepackage{epstopdf}
\usepackage{color}
\usepackage{bm}
\usepackage{comment}
\usepackage{soul}

\allowdisplaybreaks
\newtheorem{theorem}{Theorem}[section]
\newtheorem{proposition}[theorem]{Proposition}
\newtheorem{lemma}[theorem]{Lemma}
\newtheorem{corollary}[theorem]{Corollary}

\newtheorem{definition}[theorem]{Definition}

\newtheorem{thm}{Theorem}

\def\mcD{\mathcal{D}}
\def\mcE{\mathcal{E}}
\def\mcF{\mathcal{F}}
\def\msG{\mathscr{G}}

\def\USC{\mathcal{USC}}

\def\sn{\stackrel{n}{\sim}}
\def\sm{\stackrel{m}{\sim}}

\numberwithin{equation}{section}

\begin{document}
\title[Dirichlet forms on unconstrained Sierpinski carpets]{Dirichlet forms on unconstrained Sierpinski carpets}

\author{Shiping Cao}
\address{Department of Mathematics, University of Washington, Seattle, WA, 98195, U.S.A.}
\email{spcao@uw.edu}
\thanks{}

\author{Hua Qiu}
\address{Department of Mathematics, Nanjing University, Nanjing, 210093, P. R. China.}
\thanks{The research of Qiu was supported by the National Natural Science Foundation of China, grant 12071213, and the Natural Science Foundation of Jiangsu Province in China, grant BK20211142.}
\email{huaqiu@nju.edu.cn}


\date{}


\maketitle

\begin{abstract}
We construct symmetric self-similar Dirichlet forms on unconstrained Sierpinski carpets, which are natural extension of planar Sierpinski carpets by allowing the small cells to live off the $1/k$ grids. The intersection of two cells can be a line segment of irrational length, and the non-diagonal assumption is dropped in this recurrent setting.
\end{abstract}

\section{Introduction}\label{sec1}
In this paper, we study self-similar Dirichlet forms on unconstrained Sierpinski carpets ($\USC$ for short). See Figure \ref{fig1} for some examples.

\begin{figure}[htp]
	\includegraphics[width=4.9cm]{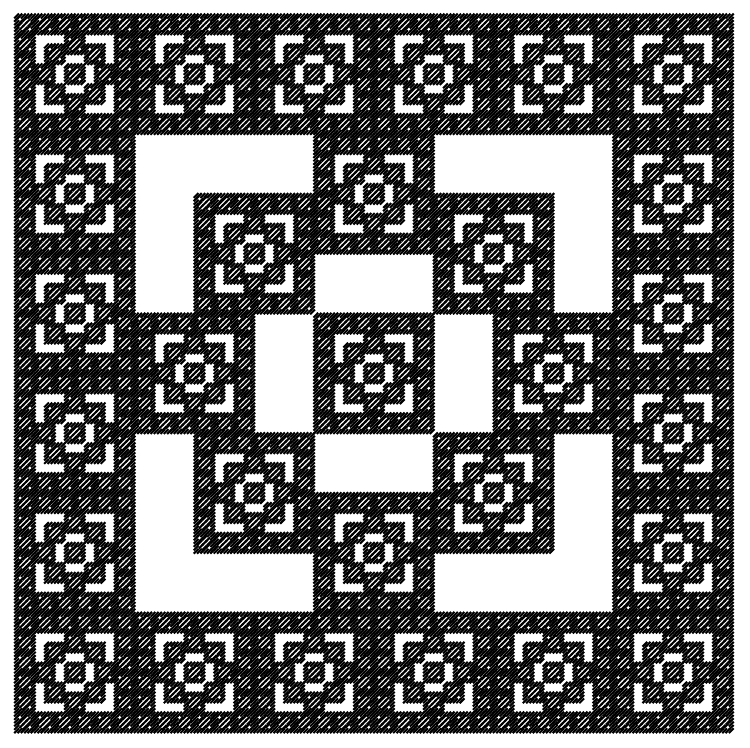}\hspace{1cm}
	\includegraphics[width=4.9cm]{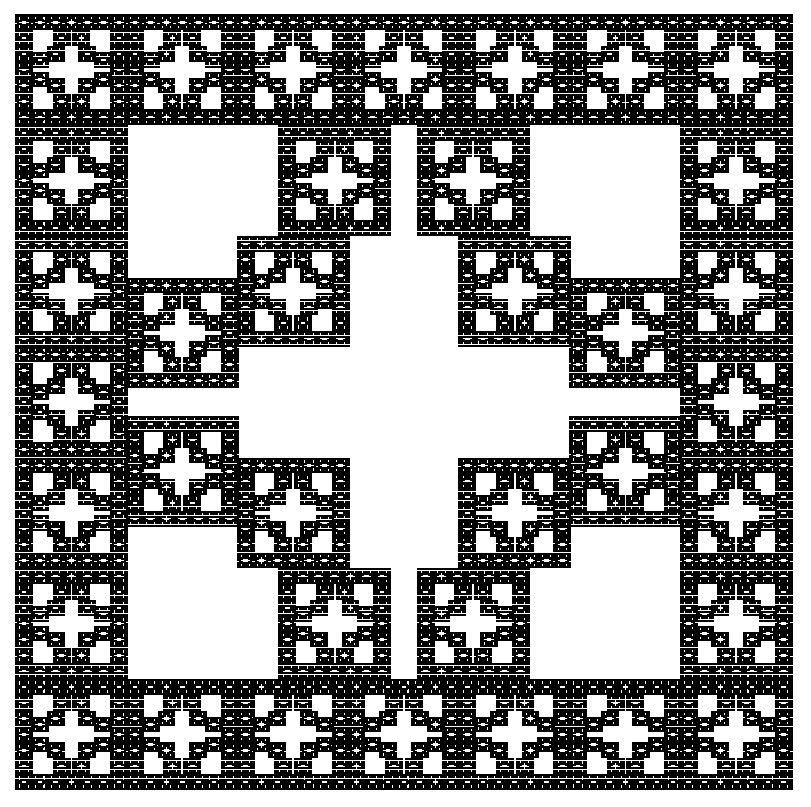}
	\caption{Unconstrained Sierpinski carpets ($\USC$).}
	\label{fig1}
\end{figure}

To see the meaning of the adjective ``\textit{unconstrained}'', let us first recall the definition of the classical Sierpinski carpets. By dividing the unit square into $k^2,k\geq 3$ identical small squares and deleting a few symmetrically (keeping all that border the boundary), we get an iterated function system (i.f.s. for short) consisting of similarities that map the unit square to a remaining small square. The attractor of the i.f.s. is a Sierpinski carpet, as long as the fractal is connected. See Figure \ref{fig2} for  the standard Sierpinski carpet, with $k=3$.

\begin{figure}[htp]
	\includegraphics[width=4.9cm]{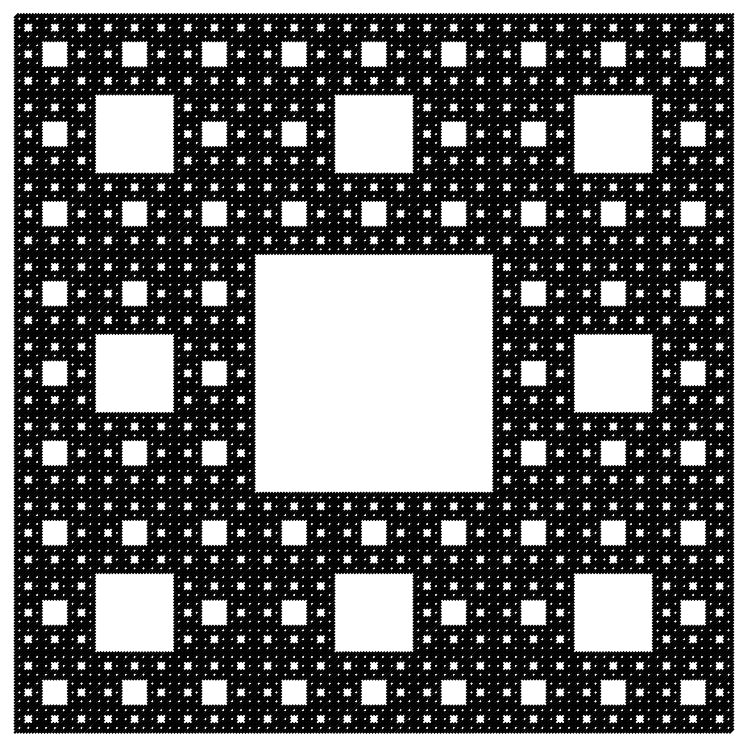}
	\caption{The standard Sierpinski carpet.}
	\label{fig2}
\end{figure}

Clearly, all the squares are constrained to lie on the $1/k$ grids in a classical Sierpinski carpet.  A larger family of fractals, which we call unconstrained Sierpinski carpets, are exactly those self-similar sets generated in a similar manner, with the constraint removed by allowing the $1/k$ sized squares to move around. See Definition \ref{def21} for precise details. We will construct and study  diffusions on these unconstrained Sierpinski carpets.\vspace{0.2cm}

The study of diffusion processes on fractals emerged as an independent research field in the  80's. Initial interest in such processes came from mathematical physicists working in the theory of disordered media \cite{AO,HBA,RTP}. On self-similar sets, the pioneering works are the constructions of  Brownian motions on the Sierpinski gasket \cite{BP,G,kus} originated by Goldstein,  Kusuoka, and Barlow and Perkins independently. The method features the study of a compatible sequence of graphs,  and is extended to post critically finite fractals \cite{ki1,ki2} by Kigami. See \cite{C,CHQES,CQ1,CQ2,FS,HMT,HN,Lindstrom,M1,M2,Pe,RT,Sabot,T} and books \cite{B,ki3,s} for the Dirichlet forms on such fractals.

The Sierpinski carpets are typical self-similar sets that are not finitely ramified. The construction of diffusions on Sierpinski carpets \cite{BB}, initiated by Barlow and Bass, is a milestone in analysis on fractals.

In the pioneering work \cite{BB}, the diffusion is defined to be a weak limit of renormalized Brownian motions on domains in $\mathbb{R}^2$, and the method was later extended to generalized Sierpinski carpets (carpets in $\mathbb{R}^d,d\geq 3$) \cite{BB3}. An alternative approach, with the language of Dirichlet forms \cite{FOT}, was given by Kusuoka and Zhou \cite{KZ}. It remained a difficult question whether the diffusions resulting from the two methods were the same until in 2010, Barlow, Bass, Kumagai and Teplyaev provided an affirmative answer \cite{BBKT}. See \cite{BB1,BB4,BB2,BB3} for more properties of the diffusions on Sierpinski carpets. Also see \cite{GY} for  recent work on the standard Sierpinski carpet based on the resistance estimate \cite{BB4} and $\Gamma$-convergence \cite{D,Mosco}.

More or less, the constructions depend on delicate structure of the Sierpinski carpets. In \cite{BB}, the non-diagonal condition is required since Brownian motion on $\mathbb{R}^2$ does not hit points. Also, the corner moves and slide moves \cite{BB,BB3} depend on the exact way that cells intersect.  In \cite{BB2}, the resistance estimate used transformations among different graphs, which are essentially generated from the local structure of the fractal. A more general framework is provided in \cite{KZ} with several kinds of Poincare constants involved. However, some extra conditions ``\textbf{(GB)}'', ``\textbf{(LS)}'' and the ``Knight moves'' method were used to complete the story.\vspace{0.2cm}

Although it may not be possible to define a good diffusion on an arbitrary fractal, it is still one of the central question in this field to see how far we can loosen the geometric restrictions \cite{B}. The goal in this paper is to extend the constructions to more general unconstrained Sierpinski carpets. In particular, we break the conditions \textbf{(GB)} and \textbf{(LS)} of \cite{KZ}. Since we are dealing with the recurrent case, the non-diagonal condition \cite{BB,BB3,BBKT} is also dropped.

\begin{thm}\label{thm1}
Let $K$ be a $\USC$ and $\mu$ be the normalized Hausdorff measure on $K$. There is a strongly local, regular, irreducible, $D_4$-symmetric, self-similar Dirichlet form $(\mcE,\mcF)$ on $L^2(K,\mu)$.
\end{thm}

See Definition \ref{def51} for the exact meaning of ``self-similar'', and also Theorem \ref{thm52} for a more elaborate statement.

Our approach is purely analytic. We are inspired by \cite{KZ}, and adopt a few results and definitions there, including the important Poincare constants. The main contributions that we make in this paper are in two aspects. First, we take advantage of strong recurrence of the electrical networks approximating the fractals, which implies that any function with small energy has small modulus of continuity, to show a side-to-side resistance estimate using the Poincare constants. In particular, we do not need the probabilistic argument of ``Knight moves'' and ``corner moves'' of Barlow and Bass \cite{BB}. Second, we use the side-to-side resistance estimate to glue functions with good boundary values together to construct bump functions with small energies. To fulfill this, we use the idea of a trace theorem by Hino and Kumagai \cite{HKtrace} to overcome the difficulty from the worse geometry of $\USC$. The existence of such bump functions will be formulated into condition \textbf{(B)} (see Section \ref{sec3}), which is the key ingredient of the construction of Dirichlet forms according to \cite{KZ}.\vspace{0.2cm}

Finally, we also point out that the unconstrained Sierpinski carpets can be viewed as moving fractals since cells are allowed to move around. In a forthcoming paper \cite{CQ4}, we will prove that the Dirichlet forms constructed in Theorem \ref{thm1} will vary continuously in a $\Gamma$-convergence sense, and the generated diffusion processes, viewed as processes in $\mathbb{R}^2$, will converge in distribution.\vspace{.2cm}

At the end, we briefly introduce the structure of the paper.

First, as the preliminary part, we introduce the precise definition and geometric properties of $\USC$ in Section \ref{sec2}.

In Section \ref{sec3}, first we introduce the Poincare constants $R_n,\lambda_n$ and $\sigma_n,n\geq 1$, from \cite{KZ}. For the convenience of readers, we slightly rewrite in Appendix \ref{AppendixA}  the proofs of several useful comparisons among these constants as illustrated in \cite{KZ}. Next, under the assumption of condition \textbf{(B)}, we provide a shorter proof of the existence of a limit form of discrete forms, based on $\Gamma$-convergence.

Section \ref{sec4} is the most important part of the paper. It contains three subsections. We postpone the proof of the side-to-side resistance estimate to Subsection \ref{sec43}, where we take advantage of strongly recurrence. In Subsections \ref{sec41} and \ref{sec42}, using the side-to-side resistance estimates, we construct functions that have small energies and good boundary values. Here ``good boundary values'' means that the functions behave like linear functions on $\mathbb{R}^2$ near the boundary, so we can easily glue these functions together to get bump functions, which will result in condition \textbf{(B)}.

Finally, we construct the self-similar Dirichlet form in Section \ref{sec5}.\\

Throughout this paper, we use $C$ and $C_1,C_2,\cdots$ to denote constants that are positive and finite. We will simply say that there is a constant $C$ or there is $C>0$. From time to time, we will write $a\asymp b$ for two functions (forms, and sequences) $a$ and $b$, if there is a constant $C>0$ so that $C^{-1}\cdot b\leq a\leq C\cdot b$. We will use the notation $a\wedge b=\min\{a,b\}$ and $a\vee b= \max\{a,b\}$.

\section{Geometry of unconstrained Sierpinski carpets}\label{sec2}
In this section, we introduce the definition of $\USC$ and prove some geometric properties. Similar as the classical Sierpinski carpets, the $\USC$s are self-similar sets generated by replacing the initial square with smaller squares and iterating.

In this paper, we consider fractals in $\mathbb{R}^2$. We will simply write $x\in \mathbb{R}^2$ for a point in $\mathbb{R}^2$, and from time to time we write $x=(x_1,x_2)$ to specify the coordinates of $x$. For two points $x,y\in \mathbb{R}^2$, we will always write
\[\overline{x,y}=\big\{(1-t)x+ty:0\leq t\leq 1\big\},\]
for the line segment connecting $x,y\in \mathbb{R}^2$. The metric $d$, throughout this paper, will always be the Euclidean metric on $\mathbb{R}^2$. In addition, we will write
\[d(A,B)=\inf_{x\in A,y\in B}d(x,y),\]
for $A,B\subset \mathbb{R}^2$, which is always positive if $A,B$ are disjoint compact sets. We also write $d(x,A)=d(\{x\},A)$ for short.

Let $\square$ be a unit square in $\mathbb{R}^2$. We let
\[q_1=(0,0),\quad q_2=(1,0),\quad q_3=(1,1),\quad q_4=(0,1)\]
be the four vertices of $\square$.
For convenience, we denote the group of self-isometries on $\square$ by
\[\msG=\big\{\Gamma_v,\Gamma_h,\Gamma_{d_1},\Gamma_{d_2},id,\Gamma_{r_1},\Gamma_{r_2},\Gamma_{r_3}\big\},\]
where $\Gamma_v,\Gamma_h,\Gamma_{d_1},\Gamma_{d_2}$ are \textit{reflections} ($v$ for vertical, $h$ for horizontal, $d_1,d_2$ for two diagonals),
\begin{equation}\label{eqn21}
\begin{aligned}
	&\Gamma_v(x_1,x_2)=(x_1,1-x_2),\qquad \Gamma_h(x_1,x_2)=(1-x_1,x_2),\\
	&\Gamma_{d_1}(x_1,x_2)=(x_2,x_1),\qquad\Gamma_{d_2}(x_1,x_2)=(1-x_2,1-x_1),
\end{aligned}
\end{equation}
 for $x=(x_1,x_2)\in \square$, $id$ is the identity mapping, and $\Gamma_{r_1},\Gamma_{r_2},\Gamma_{r_3}$ are \textit{rotations},
\begin{equation}\label{eqn22}
	\Gamma_{r_1}(x_1,x_2)=(1-x_2,x_1),\quad \Gamma_{r_2}=(\Gamma_{r_1})^2,\quad \Gamma_{r_3}=(\Gamma_{r_1})^3,
\end{equation}
around the center of $\square$ counterclockwisely with angle $\frac{j\pi}{2}$, $j=1,2,3$.
In this paper, we will focus on  structures with $\msG$-symmetry.

\begin{definition}[Unconstrained Sierpinski carpets]\label{def21}\quad
	
Let $k\geq 3$ and $4(k-1)\leq N\leq k^2-1$. Let $\{\Psi_i\}_{1\leq i\leq N}$ be a finite set of similarities with the form $\Psi_i(x)=\frac x k+ c_i$, $c_i\in\mathbb{R}^2$. Assume the following holds:

\noindent\emph{(Non-overlapping).}  $\Psi_i(\square)\cap \Psi_j(\square)$ is either a line segment, or a point, or empty, for $1\leq i,j\leq N$ with $i\neq j$.

\noindent\emph{(Connectivity). } $\bigcup_{i=1}^N \Psi_i(\square)$ is connected.

\noindent\emph{(Symmetry).} $\Gamma\big(\bigcup_{i=1}^N \Psi_i(\square)\big)=\bigcup_{i=1}^N \Psi_i(\square)$ for any $\Gamma\in \msG$.

\noindent\emph{(Boundary included).} $\overline{q_1,q_2}\subset\bigcup_{i=1}^N \Psi_i(\square)\subset\square$.

Then, call the unique compact subset $K\subset \square$ such that
\[K=\bigcup_{i=1}^N \Psi_iK\]
an \emph{unconstrained Sierpinski carpet} ($\USC$).
\end{definition}

\noindent\textbf{Remark}. To reduce the number of brackets, we write $\Psi_iK$ instead of $\Psi_i(K)$, following the notations of \cite{s}. Similarly, we will also use the notations like $\Psi_wL_i$.\vspace{0.2cm}

Here the condition $k\geq 3$ is to avoid trivial set by the symmetry condition. The condition $N\geq 4(k-1)$ is a requirement of the boundary included condition, and the condition $N\leq k^2-1$ ensures that  we are dealing with a non-trivial planar self-similar set. Note that when $k=3$, $N=8$, $K$ is the standard Sierpinski carpet. See Figure \ref{figure1} for more  examples of $\USC$.

\begin{figure}[htp]
	\includegraphics[width=4.9cm]{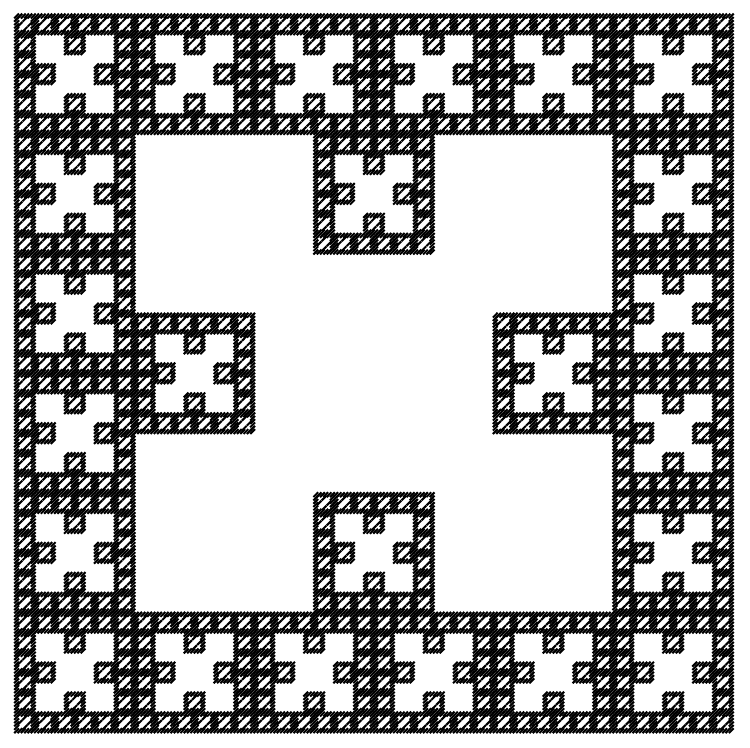}\hspace{0.2cm}
	\includegraphics[width=4.9cm]{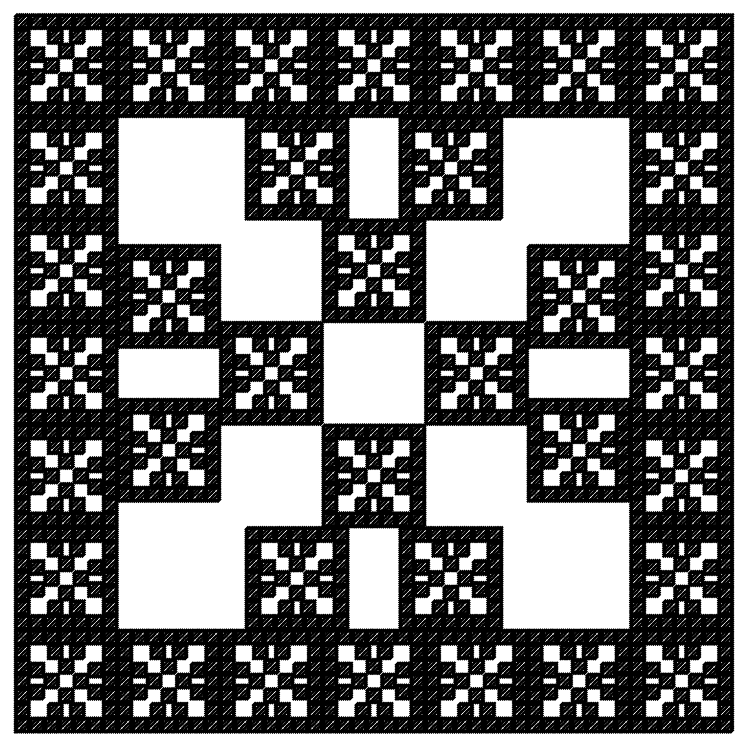}\hspace{0.2cm}
	\includegraphics[width=4.9cm]{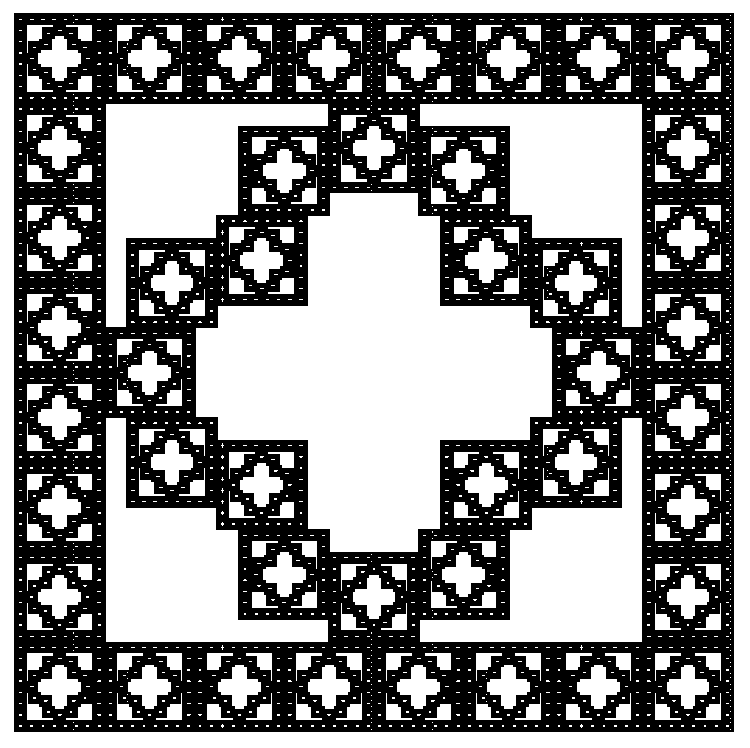}
	\caption{More unconstrained Sierpinski carpets ($\USC$).}
	\label{figure1}
\end{figure}

 Below we highlight some basic settings.\vspace{0.2cm}

\textbf{Basic settings and notations.}

 (1). Throughout the paper, $k,N$ are fixed numbers.

(2). The squares $\Psi_i(\square)$ along  the boundary of $\square$ are especially useful, and we number them in the following way. For $0\leq j\leq 3$ and $1\leq i\leq k-1$, we stipulate that
\begin{equation}\label{eqn23}
	\Psi_{(k-1)j+i}(x)=\frac{1}{k}(x-q_{j+1})+q_{j+1}+\frac{i-1}{k}(q_{j+2}-q_{j+1})
\end{equation}
with cyclic notation $q_5=q_1$.

(3). For convenience, we denote the four sides of $\square$ by
\begin{equation}
L_1=\overline{q_1,q_2},\quad L_2=\overline{q_2,q_3},\quad L_3=\overline{q_3,q_4},\quad L_4=\overline{q_4,q_1}.
\end{equation}

(4). We write $\partial_0 K=\bigcup_{i=1}^4 L_i$ for the square boundary of $K$.
\vspace{0.2cm}

In the rest of this section, we prove some geometric properties of $\USC$. Let us first introduce some more notations.

\begin{definition}\label{def22}
	Let $W_0=\{\emptyset\},W_1=\{1,2,\cdots,N\}$ be the alphabet associated with $K$,  and for $n\geq 1$, $W_n := W_1^n=\{w=w_1w_2\cdots w_n:w_i\in W_1\hbox{ for every } 1\leq i\leq n\}$ be the collection of \emph{words} of \emph{length} $n$. Also, write $W_*=\bigcup_{n=0}^\infty W_n$ for the collection of \emph{finite words}.
	
	(a). For each $w=w_1\cdots w_n\in W_n$, we denote $|w|=n$ the \emph{length} of $w$, and write
	\[\Psi_w=\Psi_{w_1}\circ \cdots \circ \Psi_{w_n}.\]
	So each $w\in W_n$ represents a \emph{$n$-cell} $\Psi_wK$ in $K$.
	
	(b). For $w, w'\in W_n$ with $n\geq 1$, we denote $w\sn w'$ if $w\neq w'$ and $\Psi_wK\cap \Psi_{w'}K\neq \emptyset$. For $A\subset W_n$, say $A$ is \emph{connected} if each pair $w\neq w'$ in $A$ is connected by a \emph{path}, i.e. there exists a chain of cells $\{w^{(i)}\}_{i=1}^l\subset A$ with $w^{(1)}=w$, $w^{(l)}=w'$ and $w^{(i)}\sn w^{(i+1)}$ for $1\leq i<l$. For disjoint connected sets $A,B\subset W_n$, we write $A\sn B$ if $w\sn w'$ for some $w\in A$ and $w'\in B$. In particular, for $w\in W_n$, we write $w\sn A$ if $w\notin A$ and $\{w\}\sn A$.

	(c). For $w=w_1\cdots w_n\in W_n,w'=w'_1\cdots w'_m\in W_m$, we write \[w\cdot w'=w_1\cdots w_n w'_1\cdots w'_m\in W_{n+m}.\]  For $A\subset W_n,B\subset W_m$, we denote $A\cdot B=\{w\cdot w':w\in A,w'\in B\}$. In particular, we abbreviate  $\{w\}\cdot B$ to $w\cdot B$.
	
	(d). For a connected set $A\subset W_n$, we write $l(A)$ the collection of real functions on $A$, and define a \emph{symmetric bilinear form} $\mcD_{n,A}$ on  $l(A)$ as,
\[\mcD_{n,A}(f,g)=\sum_{w\sn w' \text{ in } A}\big(f(w)-f(w')\big)\big(g(w)-g(w')\big)\quad \hbox{ for every } f,g\in l(A),\]
where (also in the following context) we always use $\sum_{w\sn w'}$ to mean that the sum is over all (unordered) edges $w\sn w'$, following \cite[Section 1.3]{s}.
	In particular, we write $\mcD_{n,A}(f):=\mcD_{n,A}(f,f)$, and  $\mcD_{n}:=\mcD_{n,W_n}$ for short.
	
	(e). For $n\geq 0$, we define
	\[\partial W_n=\big\{w\in W_n: \Psi_wK\cap \partial_0 K\neq \emptyset\big\}.\]
	Also, we write
	\[W_{n,i}=\big\{w\in W_n:\Psi_wK\cap L_i\neq\emptyset\big\},\text{ for } i=1,2,3,4,\]
	and $W_{*,i}=\bigcup_{n=0}^\infty W_{n,i}$.
\end{definition}

In the following, we show that a $\USC$ $K$ has the following geometric properties \textbf{(A1)}-\textbf{(A4)}. In particular, \textbf{(A1)}, \textbf{(A2)}, \textbf{(A4)} are the same as in \cite[Section 2]{KZ}, while \textbf{(A3)} is relaxed for $\USC$ (we fix the constant $c_0$ in \textbf{(A3)} for later use).

\vspace{0.2cm}

\textbf{(A1).} \emph{The \emph{open set condition}: there exists a non-empty open set $U\subset \mathbb{R}^2$ such that $\bigcup_{i=1}^N \Psi_iU\subset U$ and $\Psi_iU\cap \Psi_jU=\emptyset$ for every $i\neq j\in \{1,\cdots,N\}$. }

\textbf{(A2)}. \emph{For $f\in l(W_n)$, $\mcD_n(f)=0$ if and only if $f$ is a constant function.}

\textbf{(A3)}. \emph{There is a constant $0<c_0<1$ satisfying the following.}

\emph{(1). If $x,y\in K$ and $d(x,y)< c_0k^{-n}$, then there exist $w, w',w''\in W_n$ such that $x\in \Psi_{w}K$, $y\in \Psi_{w'}K$ and $\Psi_{w}K\cap \Psi_{w''}K\neq \emptyset$, $\Psi_{w'}K\cap \Psi_{w''}K\neq \emptyset$. }

\emph{(2). If $w,w'\in W_n$ and there is no $w''\in W_n$ so that
	\[\Psi_wK\cap \Psi_{w''}K\neq \emptyset, \quad \Psi_{w'}K\cap \Psi_{w''}K\neq \emptyset,\]
then $d(x,y)\geq c_0 k^{-n}$ for any $x\in \Psi_wK$ and $y\in \Psi_{w'}K$. }

\textbf{(A4)}. \emph{$\partial W_n\neq W_n$ for $n\geq 2$.}\vspace{0.2cm}

\noindent\textbf{Remark 1.} By \textbf{(A1)}, the \emph{Hausdorff dimension} of $K$ is $d_H=\frac{\log N}{\log k}$, the unique solution to the equation $\sum_{i=1}^N(\frac{1}{k})^\alpha=1$. Throughout the paper, we will always choose $\mu$ to be the \emph{normalized $d_H$-dimensional Hausdorff measure} on $K$. In other words, $\mu$ is the unique self-similar probability measure on $K$ such that $\mu=\frac{1}{N}\sum_{i=1}^N\mu\circ \Psi_i^{-1}$.
\vspace{0.2cm}

\noindent\textbf{Remark 2.} In \cite{KZ}, \textbf{(A3)}-(2) was stated as:

\noindent \textit{If $w,w'\in W_n$ and $\Psi_wK\cap \Psi_{w'}K=\emptyset$, then $d(x,y)\geq c_0 k^{-n}$ for any $x\in \Psi_wK$ and $y\in \Psi_{w'}K$.}

One can check that this does not hold if we let two cells $\Psi_iK,\Psi_jK$, $i\neq j\in \{1,\cdots,N\}$ to have intersection length
\[|\Psi_iK\cap \Psi_jK|=\sum_{l=2}^\infty k^{-\frac{l(l+1)}{2}},\]
where $|\Psi_iK\cap \Psi_jK|$ denotes the length of $\Psi_iK\cap \Psi_jK$. See Figure \ref{figadd1} for a concrete example.
\begin{figure}[htp]
	\includegraphics[width=4.9cm]{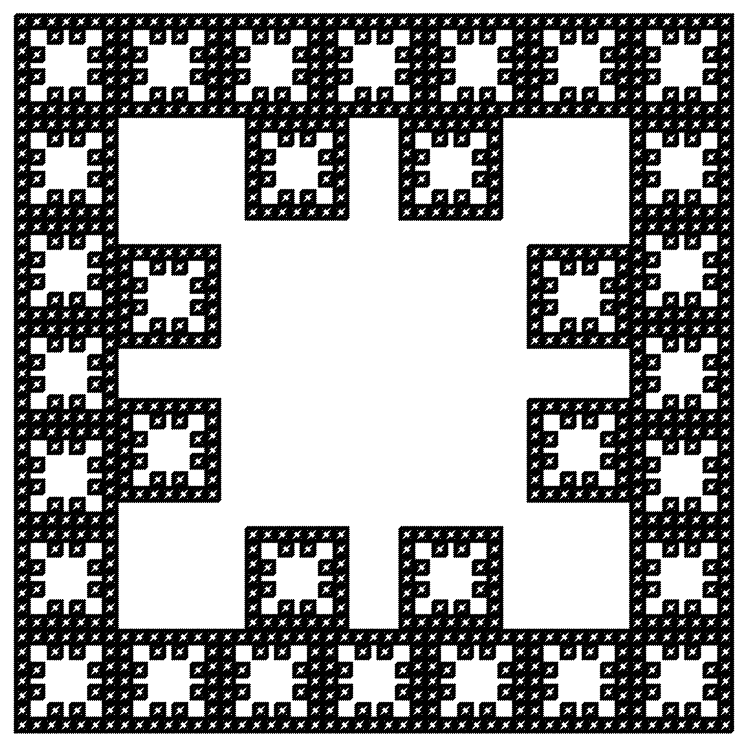}
	\caption{A $\USC$ with $k=7,N=32$ and $\Psi_{25}(x)=\frac{1}{7}x+(\frac{2}{7}+\sum\limits_{l=2}^\infty 7^{-\frac{l(l+1)}{2}},\frac{1}{7})$.}\label{figadd1}
\end{figure}

\noindent\textbf{Remark 3.} The condition \textbf{(A3)} is exactly the condition that ``\emph{the Euclidean metric is $2$-adapted}'' introduced in \cite{ki5} by Kigami.

\begin{proof}[Proof of \textbf{(A1)}-\textbf{(A4)}] \hspace{0.2cm}

\textbf{(A1)}, \textbf{(A2)}, \textbf{(A4)} are obvious. It remains to verify \textbf{(A3)}. \vspace{0.2cm}
	
Let
\[c_0=\min\Big\{\frac 12, k\cdot\min\big\{d(\Psi_iK,\Psi_jK):\Psi_iK\cap \Psi_jK=\emptyset, 1\leq i,j\leq N\big\}\Big\}.\]

\textbf{(A3)}-(1). Let $x,y\in K$ with  $d(x,y)<c_0k^{-n}$ for some $n\geq 1$.
Let \[m_0=1+\max\big\{m\geq 0: \text{ there is } \tau\in W_m \text{ such that } x,y\in \Psi_\tau K\big\}.\]

\textit{Case 1. $m_0>n$. }
	
In this case, there is $w\in W_n$ such that $x,y\in \Psi_{w}K$, so \textbf{(A3)}-(1) holds immediately.\vspace{0.2cm}

\textit{Case 2. $m_0\leq n$. }
	
In this case, we show that there is $w\neq w'\in W_{m_0}$ and $i,i'\in\{1,2,3,4\}$ such that
\[\Psi_{w}K\cap \Psi_{w'}K\neq \emptyset,\]
and
\[x\in \bigcup_{\tilde{w}\in W_{n-m_0,i}}\Psi_{w\cdot \tilde{w}}K,\quad y\in \bigcup_{\tilde{w}\in W_{n-m_0,i'}}\Psi_{w'\cdot \tilde{w}}K,\quad |i-i'|=2.\]
This makes the geometry clear, and \textbf{(A3)}-(1) follows easily. We prove the observation below.

First, by the definition of $m_0$, there is $\tau\in W_{m_0-1}$ such that $x\in \Psi_{\tau} K$, $y\in \Psi_{\tau}K$.  Next, noticing that $d(\Psi_\tau^{-1}x,\Psi_\tau^{-1}y)<c_0 k^{-n+(m_0-1)}\leq c_0k^{-1}$, by the definition of $c_0$, we have $1\leq j\neq j'\leq N$ so that $\Psi_\tau^{-1}x\in \Psi_jK$, $\Psi_{\tau}^{-1}y\in \Psi_{j'}K$ and $\Psi_jK\cap \Psi_{j'}K\neq\emptyset$. We let $w=\tau\cdot j$ and $w'=\tau\cdot j'$. Finally, choose $1\leq i,i'\leq 4$ such that
\[\Psi_wK\cap \Psi_{w'}K\subset \Psi_wL_i,\quad \Psi_wK\cap \Psi_{w'}K\subset \Psi_{w'}L_{i'},\quad |i-i'|=2.\]
If $x\notin \bigcup_{\tilde{w}\in W_{n-m_0,i}}\Psi_{w\cdot \tilde{w}}K$, then $d(x,y)\geq d(x,\Psi_{w'}K)\geq k^{-n}$, which contradicts the fact $d(x,y)<c_0k^{-n}$. So $x\in \bigcup_{\tilde{w}\in W_{n-m_0,i}}\Psi_{w\cdot \tilde{w}}K$, and by a same argument $y\in \bigcup_{\tilde{w}\in W_{n-m_0,i'}}\Psi_{w'\cdot \tilde{w}}K$.

\vspace{0.2cm}

\textbf{(A3)}-(2). We prove \textbf{(A3)}-(2) by contradiction. Let $w, w'\in W_n$, and assume there is no $w''\in W_n$ so that $\Psi_wK\cap \Psi_{w''}K\neq \emptyset,\Psi_{w'}K\cap \Psi_{w''}K\neq \emptyset$. If there are $x\in \Psi_wK,y\in \Psi_{w'}K$ so that $d(x,y)< c_0k^{-n}$, then we can find $x'\in \Psi_w(K\setminus \partial_0K)$ and $y'\in \Psi_{w'}(K\setminus \partial_0K)$ so that $d(x',y')<c_0k^{-n}$. In particular, $\Psi_wK$ and $\Psi_{w'}K$ are the only $n$-cells containing $x'$ and $y'$ respectively. By \textbf{(A3)}-(1) with respect to $x'$ and $y'$, we know there is $w''\in W_n$ satisfying $\Psi_wK\cap \Psi_{w''}K\neq \emptyset,\Psi_{w'}K\cap \Psi_{w''}K\neq \emptyset$. A contradiction.
\end{proof}

\section{A limit form of Kusuoka and Zhou}\label{sec3}
In this section, we will introduce the celebrated results of Kusuoka and Zhou \cite{KZ}. First, we define the \textit{Poincare constants} $\lambda_n$, $R_n$, $\sigma_n$, which were introduced in \cite{KZ}. In particular, we will provide a short proof of Theorem \ref{thm34}  (a statement combining of \cite[Theorem 5.4]{KZ} and \cite[Theorem 7.2]{KZ} in the recurrent case) based on the method of $\Gamma$-convergence. The method of $\Gamma$-convergence in the construction of Dirichlet forms on self-similar sets was also used by Grigor'yan and Yang in \cite{GY}. Our story will not involve the constant $\lambda_n^{(D)}$  in \cite{KZ}. Also, we modify the definition of $R_n$ for the compatibility with the weaker version of \textbf{(A3)}-(2).

\begin{definition}[Poincare constants \cite{KZ}]\label{def31}
For $n\geq 0$, $A\subset W_n$ and $f\in l(A)$, we write
\[[f]_A=(\#A)^{-1}\sum_{w\in A}f(w).\]
	
(a). For $n\geq 1$, define
\[\lambda_n=\sup\big\{\sum_{w\in W_n}\big(f(w)-[f]_{W_n}\big)^2:f\in l(W_n), \mcD_n(f)=1\big\}.\]
	
	
(b). For $n\geq 1$, $A,B\subset W_n$ with $A\cap B=\emptyset$,  define \[R_{n}(A,B)=\max\big\{\big(\mcD_n(f)\big)^{-1}:f\in l(W_n),f|_A=0,f|_B=1\big\},\] the \emph{effective resistance} between $A$ and $B$.
	
For $w\in W_*$, we write
\[\mathcal{N}_w=\bigcup\big\{w'\in W_{|w|}:\text{there is }w''\in W_{|w|}\text{ so that }\Psi_wK\cap \Psi_{w''}K\neq \emptyset, \Psi_{w'}K\cap \Psi_{w''}K\neq \emptyset\big\},\]
call it the \emph{$2$-adapted neighbourhood} of $w$ (see \cite[Chapter 2]{ki5} for the adaptedness), and abbreviate $W_{|w|}\setminus \mathcal{N}_w$ to $\mathcal {N}_w^c$.
	
For $m\geq 1$, define
\[R_m=\inf\big\{R_{{|w|+m}}(w\cdot W_m,  \mathcal{N}_w^c\cdot W_m):w\in W_*\setminus \{\emptyset\}\big\}.\]
	
(c). For $m,n\geq 1$ and $w\sn w'$,  define
\[\sigma_{m}(w,w')=\sup\big\{N^m\big([f]_{w\cdot W_m}-[f]_{w'\cdot W_m}\big)^2:f\in l(W_{n+m}), \mcD_{n+m, \{w,w'\}\cdot W_m}(f)=1\big\}.\]
	
For $m\geq 1$, define
\[\sigma_m=\sup\big\{\sigma_{m}(w,w'):n\geq 1, w,w'\in W_n,w\sn w'\big\}.\]
\end{definition}

\noindent\textbf{Remark 1}. We will use effective resistances a lot of times in this paper, so we review some basic facts here for the convenience of readers. Readers can also read the monographs \cite{ki3,ki4} for deep studies on effective resistances on general spaces. For $n\geq 1$ and $A,B\subset W_n$ with $A\cap B=\emptyset$, the following two claims hold.
	
(1).  We can find a unique $h\in l(W_n)$, which is partially reflected in the definition (we use ``$\max$'' instead of ``$\sup$''), such that $h|_A=0$, $h|_B=1$ and $\mcD_n(h)=\big(R_n(A,B)\big)^{-1}$. In fact, $\{f\in l(W_n):f|_A=0\}$ equipped with the inner product $\mcD_n(\cdot,\cdot)$ is a finitely dimensional Hilbert space, and $\{f\in l(W_n):f|_A=0,f|_B=1\}$ is a convex closed subset (since it is an affine subspace of a finitely dimensional Hilbert space).

It follows immediately that $R_n(A,B)$ is positive and finite.

(2). If $f\in l(W_n)$ satisfies that $f|_A\leq a$ and $f|_B\geq b$ for some $a<b$, then $\mcD_n(f)\geq (b-a)^2\big(R_n(A,B)\big)^{-1}$. In fact, let $g=\big(\frac{f-a}{b-a}\vee 0\big)\wedge 1$, then $g|_A=0$, $g|_B=1$ and by the Markov property, we have
\[
\mcD_n(f)=(b-a)^2\mcD_n(\frac{f-a}{b-a})\geq (b-a)^2\mcD_n(g)\geq (b-a)^2\big(R_n(A,B)\big)^{-1}.
\]

Finally, we remark that we need to consider the effective resistances with respect to $\mcD_{n,A'}$ for some connected $A'\subset W_n$, for example, in Lemma \ref{lemma47}. The above observations, with the definition of effective resistances adjusted accordingly still work there. \vspace{0.2cm}

\noindent\textbf{Remark 2}. Similarly to Remark 1 (1), for $m,n\geq 1$ and $w\sn w'$, we actually have
\[\sigma_m(w,w')=\max\big\{N^m\big([f]_{w\cdot W_m}-[f]_{w'\cdot W_m}\big)^2:f\in l(W_{n+m}), \mcD_{n+m, \{w,w'\}\cdot W_m}(f)=1\big\},\]
i.e. the supremum is attained. In fact, if we fix a word $v\in W_m$, then
\[\begin{aligned}
\sigma_m(w,w')=N^m\Big(\inf\big\{\mcD_{n+m, \{w,w'\}\cdot W_m}(f): &f\in l(\{w,w'\}\cdot W_m),\\& [f]_{w\cdot W_m}-[f]_{w'\cdot W_m}=1,\ f(w\cdot v)=0\big\}\Big)^{-1}.
\end{aligned}\]
Then, just as in Remark 1 (1), $\{f\in l(\{w,w'\}\cdot W_m):f(w\cdot v)=0\}$ equipped with the inner product $\mcD_{n+m,\{w,w'\}\cdot W_m}(\cdot,\cdot)$ is a Hilbert space, and $\{f\in l(\{w,w'\}\cdot W_m):f(w\cdot v)=0,[f]_{w\cdot W_m}-[f]_{w'\cdot W_m}=1\}$ is a convex closed subset.\vspace{0.2cm}

One of the important results in \cite{KZ}, based on the conditions \textbf{(A1)}-\textbf{(A4)}, is the comparison of the Poincare constants. See Proposition \ref{prop32} for the inequalities.

In particular, in our setting, \textbf{(A3)}-(2) is a little weaker than the original one, but the proof in \cite{KZ} still works perfectly. For the convenience of readers, we reproduce the proof of Proposition \ref{prop32} and simplify a few detailed arguments in Appendix \ref{AppendixA}.

\begin{proposition}[\cite{KZ}, Theorem 2.1]\label{prop32}
For a fixed $\USC$ $K$, there is a constant $C>0$ such that
\begin{equation}\label{eqn31}
C^{-1}\cdot \lambda_nN^{m}R_m\leq \lambda_{n+m}\leq C\cdot \lambda_n\sigma_m \quad\hbox{ for every } n\geq 1,m\geq 1,
\end{equation}
and,
\begin{equation}\label{eqn32}
R_n\geq C\cdot(k^2N^{-1})^n\quad \hbox{ for every } n\geq 1.
\end{equation}
In addition, all the constants $\lambda_n,R_n$ and $\sigma_n$, $n\geq 1$ are positive and finite.
\end{proposition}

\noindent\textbf{Remark.} Proposition \ref{prop32} consists of (2.3), (2.4) and (2.6) of \cite[Theorem 2.1]{KZ}. \vspace{0.2cm}

Properties of $R_n$ were extensively explored in \cite{ki5} by Kigami for compact metric spaces. See \cite[Lemma 4.6.15]{ki5} for a generalized version of (\ref{eqn32}). In particular, (\ref{eqn32}) implies the process is recurrent on an infinite Sierpinski graph since $N<k^2$. It is  not hard to see from (\ref{eqn31}) that
\begin{equation}\label{eqn33}
C^{-1}\cdot N^mR_m\leq \lambda_m\leq C\cdot \sigma_m,
\end{equation}
for some $C>0$ independent of $m$. In fact, by letting $m=1$ in (\ref{eqn31}), one immediately see $\lambda_n\asymp \lambda_{n+1}$. Then, by letting $n=1$ in (\ref{eqn31}), we get (\ref{eqn33}). To see that $N^mR_m\asymp\lambda_m\asymp \sigma_m$, we need to verify another inequality, stated in the following condition \textbf{(B)}.
\vspace{0.2cm}

\textbf{(B)}. \textit{There is a constant $C>0$ such that $\sigma_n\leq C\cdot N^nR_n$, for $n\geq 1$. }
\vspace{0.2cm}

\begin{lemma}\label{lemma33}
Assume \textbf{(B)}, then there is $r>0$ such that
	\[N^{-n}\lambda_n\asymp r^{-n}.\]
	In addition,  $N^{n}R_n\asymp \lambda_n\asymp \sigma_n$ and $r\leq\frac{N}{k^2}$.
\end{lemma}
\begin{proof}
	The estimate $N^{n}R_n\asymp\lambda_n\asymp \sigma_n$ is an immediate consequence of (\ref{eqn33}) and \textbf{(B)}. Then, by (\ref{eqn31}), there are constants $C_1,C_2>0$ so that
	\[C_1\cdot\lambda_n\lambda_m\leq \lambda_{n+m}\leq C_2\cdot\lambda_n\lambda_m\quad \hbox{ for every } n,m\geq 1.\]
	It follows by a routine argument (see \cite{BB4,KZ}) that there exists $r>0$ such that $\lambda_n\asymp N^nr^{-n}$. Finally, by (\ref{eqn32}), $N^{-n}\lambda_n\asymp R_n\geq C\cdot (k^2N^{-1})^n$, so $r\leq\frac{N}{k^2}$.
\end{proof}

We call $r$ the \textit{renormalization factor} in this paper. For $\USC$, we always have $r\leq 1-\frac{1}{k^2}$.\vspace{0.2cm}

With some abuse of notations, we write
\[\mcD_n(f)=\mcD_n(P_nf)\quad \hbox{ for every } f\in L^1(K,\mu),\]
where $P_n:L^1(K,\mu)\to l(W_n)$ is defined as
\begin{equation}\label{eqn34}
	(P_nf)(w)=N^{n}\int_{\Psi_wK}f(x)\mu(dx) \quad\hbox{ for every } w\in W_n.
\end{equation}
Then, $\mcD_n$, $n\geq 1$ can be viewed as \emph{quadratic forms} on $L^2(K,\mu)$.

In \cite{KZ}, it is shown that a limit form of the discrete energies $\mcD_n$ can be well defined under slightly different conditions (see \textbf{(B1)}, \textbf{(B2)} of \cite{KZ}). The same result still holds here, and the essences of the arguments are still the same as in \cite{KZ}. We combine \cite[Theorem 5.4]{KZ} and \cite[Theorem 7.2]{KZ} together into the following Theorem \ref{thm34}.

\begin{theorem}\label{thm34}
Assume \textbf{(B)} and $0<r<1$. Let
\[\bar{\mcF}=\{f\in L^2(K,\mu):\sup_{n\geq 1}r^{-n}\mcD_n(f)<\infty\}.\]
	
(a). Let $\theta=-\frac{\log r}{\log k}$. Then there is a constant $C>0$ such that
\[|f(x)-f(y)|^2\leq C\cdot d(x,y)^{\theta}\sup_{n\geq 1}r^{-n}\mcD_n(f) \quad\text{ for every }  x,y\in K,\, f\in \bar{\mcF}.\]
In particular, $\bar{\mathcal F}\subset C(K)$.
	
(b). There is a  regular Dirichlet form $(\bar{\mcE},\bar{\mcF})$ and $C_1,C_2>0$ such that
\[C_1\cdot \sup_{n\geq 1} r^{-n}\mcD_n(f)\leq \bar{\mcE}(f)\leq C_2\cdot \liminf_{n\to\infty} r^{-n}\mcD_n(f)\quad
\text{ for every } f\in \bar{\mcF}.\]
Moreover, $(\bar{\mcE},\bar{\mcF})$ is $D_4$-\emph{symmetric}, i.e. for every $f\in \mcF$ and $\Gamma\in \mathscr{G}$, we have $f\circ\Gamma\in\bar\mcF$ and $\bar\mcE(f)=\bar\mcE(f\circ\Gamma)$, where $\mathscr{G}$ is the group of self-isometrics on $\square$ as defined in Section \ref{sec2}.
\end{theorem}

Theorem \ref{thm34} was proved with a probabilistic approach in \cite{KZ}. In the rest of this section, we provide an alternative  approach to prove Theorem \ref{thm34} by using \emph{the method of $\Gamma$-convergence}, but postpone the difficult verification of the condition \textbf{(B)} to the next section.

The alternative short proof \big(especially the proof of regularity of $(\bar{\mcE},\bar{\mcF})$\big) only works in the recurrent situation. In other words, we need to use $r<1$. Barlow-Bass's work \cite{BB3} is the only successful case of construction of non-recurrent Dirichlet forms on self-similar sets.
In particular, we will use the following observation, which comes from \cite[Theorem 7.2]{KZ}.

\begin{lemma}\label{lemma35}
There is $C>0$ such that
\[\big|f(w\cdot w')-[f]_{w\cdot W_n}\big|^2\leq C\cdot N^{-n}\lambda_n\mcD_{m+n,w\cdot W_n}(f),\]
for any $m\geq 0, w\in W_m$, $n\geq 1,w'\in W_n$ and $f\in l(W_{m+n})$.
\end{lemma}
\begin{proof}
	For short, we define $g\in l(W_n)$ by $g(v)=f(w\cdot v)$ for every $v\in W_n$. Let $w'=w'_1\cdots w'_n$ be the word that appears in the lemma. Let $B_n=W_n$ and $B_l=w'_1\cdots w'_{n-l}\cdot W_l$ for $0\leq l\leq n-1$. By applying the definition of $\lambda_l$, $l\geq 1$, we have
	\[
	\big|[g]_{B_{l-1}}-[g]_{B_{l}}\big|^2 \leq N^{-l+1}\sum_{v\in B_{l-1}}\big(g(v)-[g]_{B_{l}}\big)^2\leq N^{-l+1}\lambda_l\mcD_n(g).
	\]
	On the other hand, by using the left side of (\ref{eqn31}) and  (\ref{eqn32}), we have
	\[N^{-l}\lambda_l\leq N^{-l}\cdot C_1\frac{\lambda_n}{N^{n-l}R_{n-l}}=C_1\cdot \frac{N^{-n}\lambda_n}{R_{n-l}}\leq C_2\cdot(\frac{N}{k^2})^{n-l}(N^{-n}\lambda_n),\]
	for some $C_1,C_2>0$. Thus by taking $C=\frac{C_2N}{(1-k^{-1}N^{1/2})^2}$, we have
	\[\big|g(w')-[g]_{W_n}\big|\leq \sum_{l=1}^n  \big|[g]_{B_{l-1}}-[g]_{B_{l}}\big|\leq \sum_{l=1}^{n} \sqrt{N^{-l+1}\lambda_{l}}\cdot\sqrt{\mcD_n(g)}\leq \sqrt{C}\cdot \sqrt{N^{-n}\lambda_n}\cdot\sqrt{\mcD_n(g)}.
    \]
	The lemma then follows because $\big|f(w\cdot w')-[f]_{w\cdot W_n}\big|=\big|g(w')-[g]_{W_n}\big|$ and $\mcD_{n}(g)=\mcD_{m+n,w\cdot W_n}(f)$.
\end{proof}

We also need the following easy observation.

\begin{lemma}\label{lemma36}
	Let $n,m\geq 1$ and $f\in l(W_{n+m})$. Let $f'\in l(W_n)$ such that $f'(w)=[f]_{w\cdot W_m}$ for every $w\in W_n$. Then we have
	\[\mcD_{n}(f')\leq 8 \cdot N^{-m}\sigma_m\mcD_{n+m}(f).\]
\end{lemma}
\begin{proof}
	By the definition of $\sigma_m$, for any $w\sn w'$, we see that
	\[\big|f'(w)-f'(w')\big|^2\leq N^{-m}\sigma_m\mcD_{n+m,\{w,w'\}\cdot W_m}(f).\] Taking the summation over all $w\sn w'$, $w,w'\in W_n$, noticing that each $w$ has at most $8$ neighbours,  we get the desired inequality.
\end{proof}

\noindent\textbf{Remark.} This lemma is also useful in proving Proposition \ref{prop32}, see Appendix {\ref{AppendixA}}.\vspace{0.2cm}

Before we prove Theorem \ref{thm34}, now we briefly recall the concept of \textit{$\Gamma$-convergence}. To see general results on $\Gamma$-convergence, please refer to  the book \cite{D}.

The concept of $\Gamma$-convergence can be defined on general topological spaces (see \cite[Definition 4.1]{D}). In particular, on topological spaces with the first axiom of countability, there is an equivalent characterization that is easier to use (see \cite[Proposition 8.1]{D}). We state it below in Definition \ref{def37}.

\begin{definition}\label{def37} Let $(X,\tau)$ be a topological space satisfying the first axiom of countability. For functions $f_n, n\geq 1$ and $f$ from $X$ to $\mathbb{R}\cup \{-\infty,\infty\}$, we say $f_n$ \emph{$\Gamma$-converges} to $f$ if the following holds.
	
	(a). For any $x\in X$ and for any sequence $x_n$ converging to $x$ in $X$,
	\[f(x)\leq\liminf\limits_{n\to\infty}f_n(x_n). \]
	
	(b). For any $x\in X$, there is a sequence $x_n$ converging to $x$ in $X$ such that
	\[f(x)=\lim\limits_{n\to\infty} f_n(x_n).\]
\end{definition}

In this paper, we are interested in non-negative symmetric bilinear forms $(\mcE,\mcF)$ on $L^2(K,\mu)$, where the linear subspace $\mcF\subset L^2(K,\mu)$ is called the domain, $\mcE$ maps from $\mcF\times\mcF:=\{(f,g): f,g\in\mathcal F\}$ to $\mathbb{R}$. We say that the bilinear form $(\mcE,\mcF)$ is \textit {closed} if $\mcF$ is a Hilbert space equipped with the inner product $\mcE_1(f,g)=\mcE(f,g)+\int_K f(x)g(x)\mu(dx)$.

To apply $\Gamma$-convergence, it's more natural to consider the quadratic form (with extended real values) associated with the symmetric bilinear form $(\mcE,\mcF)$ (see \cite[Definition 11.7]{D}):
\[\mcE(f)=
\begin{cases}
	\mcE(f,f),&\text{ if }f\in \mcF,\\
	\infty,&\text{ if }f\notin \mcF,
\end{cases}\]
where we still use $\mcE$ to denote the quadratic form $L^2(K,\mu)\to [0,\infty]$ by a little abuse of the notation. Conversely, a functional $\mcE:L^2(K,\mu)\to [0,\infty]$ is a \emph{non-negative quadratic form} (with extended real values) if and only if\\
(a). $\mcE(0)=0$; \\
(b). $\mcE(tf)=t^2\mcE(f)$,  for any $f\in L^2(K,\mu)$ and any $t\in \mathbb{R}$;\\
(c). $\mcE(f+g)+\mcE(f-g)=2\mcE(f)+2\mcE(g)$ for any $f,g\in L^2(K,\mu)$.\\
Given such a functional, one can define a \emph{non-negative symmetric bilinear form} $(\mcE,\mcF)$ by $\mcF=\{f\in L^2(K,\mu):\mcE(f)<\infty\}$ and $\mcE(f,g)=\frac{1}{4}\big(\mcE(f+g)-\mcE(f-g)\big)$. So there is a simple one-to-one correspondence between quadratic forms and symmetric bilinear forms. For a proof of this fact, please refer to \cite[Proposition 11.9]{D} (where a stronger result is proved).

A crucial observation is that $(\mcE,\mcF)$ is closed if and only if the associated quadratic form $\mcE$ is \emph{lower-semicontinuous}. See \cite[Proposition 12.16]{D} for a proof. \vspace{0.2cm}

By some well known properties of $\Gamma$-convergence (See \cite[Proposition 6.8, Theorems 8.5 and 11.10]{D}), we know that  $\Gamma$-convergence on $L^2(K,\mu)$ is weak in the following sense.

\begin{proposition}\label{thm38}
	Let $\mcE_n$, $n\geq 1$  be a sequence of non-negative lower-semicontinuous quadratic forms on $L^2(K,\mu)$ (with extended real values). There is a subsequence $n_l$, $l\geq 1$ and a lower-semicontinuous non-negative quadratic form (with extended real values) $\mcE$ on $L^2(K,\mu)$ such that $\mcE_{n_l}$ $\Gamma$-converges to $\mcE$.
\end{proposition}

\begin{proof}[Proof of Theorem \ref{thm34}]
	First, we introduce the operators $Q_n,n\geq 0$, which will be useful for proving both (a) and (b). For $n\geq 0$, let
	\[\mathcal{M}_n=\big\{f\in L^2(K,\mu):f(x)=P_nf(w), \text{ for any }w\in W_n \text{ and a.e. }x\in \Psi_wK\big\},\]
	which is the space of functions whose restriction on each $n$-cell is a.e. a constant. Define $Q_n$ to be the orthogonal projection from $L^2(K,\mu)$ to $\mathcal{M}_n$, and we choose a nice version (for the proof of (a)) as follows:
\[Q_nf(x)=\max_{w\in W_n:x\in \Psi_wK} P_nf(w)\quad \hbox{ for every } x\in K.\]
It is routine to see that $Q_nf$ converges to $f$ in $L^2(K,\mu)$. In fact, if $f\in C(K)$, then $Q_nf$ converges uniformly to $f$, hence converges to $f$ in $L^2(K,\mu)$. For general $f\in L^2(K,\mu)$, for any $\varepsilon>0$, we choose $f'\in C(K)$ so that $\|f'-f\|_{L^2(K,\mu)}<\frac{\varepsilon}{3}$, and then for $n$ large enough $\|f-Q_nf\|_{L^2(K,\mu)}\leq \|f-f'\|_{L^2(K,\mu)}+\|f'-Q_nf'\|_{L^2(K,\mu)}+\|Q_nf-Q_nf'\|_{L^2(K,\mu)}<\varepsilon$. \vspace{0.2cm}

(a). For $f\in \bar{\mcF}$, we can see that $Q_nf$ converges uniformly. In fact, for each $x\in K$ and $m\geq 0$, $n\geq 1$, by Lemma \ref{lemma35} and \ref{lemma33}, we have
\begin{equation}\label{eqn3add1}
\begin{aligned}
\big |Q_{m+n}f(x)-Q_mf(x)\big|
\leq &\max_{w\in W_m,v\in W_{m + n}:x\in \Psi_{w}K\cap \Psi_v K}\big|P_{m+n}f(v)-P_mf(w)\big|\\
\leq &\max_{w\in W_m,\tau\in W_{n}:x\in \Psi_{w\cdot \tau}K}\big|P_{m+n}f(w\cdot \tau)-P_mf(w)\big|\\
&+ \max_{v, v'\in W_{m + n}: x\in \Psi_v K\cap \Psi_{v'}K}\big|P_{m+n}f(v)-P_{m+n}f(v')\big|\\
\leq &C_1\cdot \sqrt{N^{-n}\lambda_n\mathcal{D}_{m+n}(f)} + \sqrt{\mathcal{D}_{m + n}(f)}\\
\leq &C_2\cdot \sqrt{r^{-n}\mathcal{D}_{m+n}(f)}
\leq C_2\cdot \sqrt{r^m\sup_{n'\geq 1}r^{-n'}\mcD_{n'}(f)},
\end{aligned}
\end{equation}
for some constants $C_1, C_2>0$. Since $Q_nf$ converges to $f$ in $L^2(K,\mu)$, the uniform limit of $Q_nf$ is a version of $f$. So in the following we assume $f(x)=\lim\limits_{n\to\infty}Q_nf(x)$ for every $x\in K$. By letting $n\to\infty$ in (\ref{eqn3add1}), we get
\begin{equation}\label{eqn3add2}
\big|f(x)-Q_mf(x)\big|\leq C_2\cdot \sqrt{r^m\sup_{n\geq 1}r^{-n}\mcD_{n}(f)}\quad \hbox{ for every } x\in K.
\end{equation}
Now, consider $x\neq y\in K$ such that $d(x,y)<c_0 k^{-1}$ (Recall the constant $c_0$ in \textbf{(A3)}). Choose $m\geq 1$ so that $c_0k^{-m-1}\leq d(x,y)< c_0k^{-m}$, and $w,w'\in W_m$ so that $x\in \Psi_wK,Q_mf(x)=P_mf(w)$ and $y\in \Psi_{w'}K, Q_mf(y)=P_mf(w')$. Then, by \textbf{(A3)}, there is $w''\in W_m$ such that $\Psi_{w''}K\cap\Psi_wK\neq \emptyset$ and $\Psi_{w''}K\cap\Psi_{w'}K\neq \emptyset$, whence
\begin{equation}\label{eqn3add3}
\begin{aligned}
\big |Q_mf(x)-Q_mf(y)\big|&=\big|P_mf(w)-P_mf(w')\big |
\\&\leq \big|P_mf(w)-P_mf(w'')\big|+\big|P_mf(w'')-P_mf(w')\big|\\
&\leq 2\sqrt{\mcD_m(f)}\leq 2\cdot\sqrt{r^m\sup_{n\geq 1}r^{-n}\mcD_n(f)}.
\end{aligned}
\end{equation}
Combining (\ref{eqn3add2}), (\ref{eqn3add3}) and the fact $d(x,y)\geq c_0k^{-m-1}$, one immediately has
\[\big|f(x)-f(y)\big|^2\leq C_3\cdot r^m\sup_{n\geq 1}r^{-n}\mcD_n(f)\leq C_4\cdot d(x,y)^{-\frac{\log r}{\log k}}\sup_{n\geq 1}r^{-n}\mcD_n(f)\]
for some constants $C_3, C_4>0$ as desired. Finally, if $d(x,y)\geq c_0k^{-1}$, we have $|f(x)-f(y)|^2\leq C_5\cdot \sup_{n\geq 1}r^{-n}\mcD_n(f)$ for some $C_5>0$ by simply using (\ref{eqn3add2}) with $m=0$. \vspace{0.2cm}
	
(b). By Proposition \ref{thm38}, there is a subsequence $n_l$, $l\geq 1$ and a limit form $\bar{\mcE}$ on $L^2(K,\mu)$ so that $r^{-n_l}\mcD_{n_l}$ $\Gamma$-converges to $\bar{\mcE}$.
	
We now show that
\begin{equation}\label{eqn35}
C_6\cdot \sup_{n\geq 1} r^{-n}\mcD_n(f)\leq \bar{\mcE}(f)\leq C_7\cdot \liminf_{n\to\infty} r^{-n}\mcD_n(f)\quad\hbox{ for every } f\in L^2(K,\mu),
\end{equation}
for some constants $C_6,C_7>0$ independent of $f$. This will rely on the fact that $N^{-m}\sigma_m\asymp r^{-m}$,
which is a consequence of Lemma \ref{lemma33}. Clearly, by the definition of $\Gamma$-convergence, we have
\[\bar{\mcE}(f)\leq \liminf_{l\to\infty}r^{-n_l}\mcD_{n_l}(f)\leq 8C_8\cdot \liminf_{n\to\infty}r^{-n}\mcD_{n}(f),\]
where the second inequality is due to Lemma \ref{lemma36} and the fact that $r^{-n_l}\leq C_8\cdot r^{-n}(N^{n-n_l}\sigma^{-1}_{n-n_l})$ for some $C_8>0$ and all $n>n_l$.  For the other direction, we pick a sequence $f_l$, $l\geq 1$ that converges to $f$ in $L^2(K,\mu)$, and $\bar{\mcE}(f)=\lim\limits_{l\to\infty}r^{-n_l}\mcD_{n_l}(f_l)$. Then,
	\[\bar{\mcE}(f)=\lim\limits_{l\to\infty}r^{-n_l}\mcD_{n_l}(f_l)\geq \frac 18C_9\cdot \lim\limits_{l\to\infty}r^{-n}\mcD_n(f_l)=\frac18C_9\cdot r^{-n}\mcD_n(f)\quad \hbox{ for every } n\geq 1,\]
where the inequality is due to Lemma \ref{lemma36} and the fact that $r^{-n_l}\geq C_9\cdot r^{-n}N^{-n_l+n}\sigma_{n_l-n}$ for some $C_9>0$ and all $n_l>n$. Thus, we get (\ref{eqn35}).
	
Clearly $\bar{\mcF}=\{f\in L^2(K,\mu):\bar\mcE(f)<\infty\}$ by (\ref{eqn35}), and $\bar{\mcE}$ induces a non-negative symmetric bilinear form $(\bar{\mcE},\bar{\mcF})$ on $L^2(K,\mu)$. In addition, $(\bar{\mcE},\bar{\mcF})$ is closed as $\bar{\mcE}$ is lower-semicontinuous. It remains to show the Markov property, regularity and $D_4$-symmetry of $(\bar\mcE,\bar{\mcF})$. Then it follows from the regularity of $(\bar{\mcE},\bar{\mcF})$ that $\bar{\mcF}$ is dense in $L^2(K,\mu)$, so $(\bar{\mcE},\bar{\mcF})$ is a regular Dirichlet form.\vspace{0.2cm}
	
\textit{Markov property.} Let $f\in L^2(K,\mu)$, we need to show $\bar{\mcE}(\bar{f})\leq \bar{\mcE}(f)$ for $\bar{f}=(f\vee 0)\wedge 1$.

First, we will find a sequence $f_l\in \mathcal{M}_{n_l}$, $l\geq 1$ such that $f_l\to f$ in $L^2(K,\mu)$ and $\bar{\mcE}(f)=\lim\limits_{l\to\infty}r^{-n_l}\mcD_{n_l}(f_l)$. By $\Gamma$-convergence, there is a sequence $g_l\in L^2(K,\mu)$ so that $g_l\to f$ in $L^2(K,\mu)$ and $\bar{\mcE}(f)=\lim\limits_{l\to\infty}r^{-n_l}\mcD_{n_l}(g_l)$. It suffices to choose $f_l=Q_{n_l}g_l,l\geq 1$. In fact, we have $\mcD_{n_l}(f_l)=\mcD_{n_l}(g_l)$ immediately, and in addition, $f_l\to f$ in $L^2(K,\mu)$ is the consequence of the estimate
\[\begin{aligned}
\|f_l-f\|_{L^2(K,\mu)}=\|Q_{n_l}g_l-f\|_{L^2(K,\mu)}&\leq \|Q_{n_l}g_l-Q_{n_l}f\|_{L^2(K,\mu)}+\|Q_{n_l}f-f\|_{L^2(K,\mu)}\\
&\leq \|g_l-f\|_{L^2(K,\mu)}+\|Q_{n_l}f-f\|_{L^2(K,\mu)},
\end{aligned}\]
and the fact that $Q_nf\to f$ in $L^2(K,\mu)$ as $n\to\infty$.

Then, letting $\bar{f}_l=(f_l\vee 0)\wedge 1$, we have $\bar{f}_l\to \bar{f}$ in $L^2(K,\mu)$ as $\|\bar{f}_l-\bar{f}\|_{L^2(K,\mu)}\leq \|f_l-f\|_{L^2(K,\mu)}\hbox{ for every }l\geq 1$, and hence
\[\bar{\mcE}(\bar{f})\leq \liminf_{l\to\infty}r^{-n_l}\mcD_{n_l}(\bar{f}_l)\leq \liminf_{l\to\infty}r^{-n_l}\mcD_{n_l}(f_l)=\bar{\mcE}(f),\]
where we use the Markov property of $\mcD_{n_l}$ on $\mathcal{M}_{n_l}$.\vspace{0.2cm}
	
\textit{Regularity}. Since $\bar{\mcF}\subset C(K)$ by (a), it remains to show $\bar{\mcF}$ is dense in $C(K)$. We use the Stone-Weierstrass theorem of real valued functions on compact Hausdorff spaces (see \cite[Theorem 2.4.11]{Dr} for a rudimentary proof). We need to verify that $\bar{\mcF}$ is an algebra contained in $C(K)$ such that $\bar{\mcF}$ contains constants and separates points in $K$.
	
First, $\bar{\mcF}$ is an algebra. $\bar{\mcF}$ is a linear space, so we only need to show that $f\in \bar{\mcF},g\in \bar{\mcF}$ implies $f\cdot g\in \bar{\mcF}$. This is a consequence of the fact that $\bar{\mcF}\subset C(K)$, which implies that $Q_n(f)\cdot Q_n(g)\to f\cdot g$ in $L^2(K,\mu)$ (the convergence is also uniform indeed) and $\mcD_n\big(Q_n(f)\cdot Q_n(g)\big)\leq 2\|f\|^2_{L^\infty(K,\mu)} \mcD_n(g)+2\|g\|^2_{L^\infty(K,\mu)}\mcD_n(f)$, where in the second formula we use the observations $\|Q_nf\|_{L^\infty(K,\mu)}\leq \|f\|_{L^\infty(K,\mu)}$, $\|Q_ng\|_{L^\infty(K,\mu)}\leq \|g\|_{L^\infty(K,\mu)}$, $P_n\big(Q_n(f)\cdot Q_n(g)\big)=P_n\big(Q_n(f)\big)\cdot P_n\big(Q_n(g)\big)$ and the inequality $\sqrt{\mcD_n(u\cdot v)}\leq \|u\|_{l^\infty}\sqrt{\mcD_n(v)}+\|v\|_{l^\infty}\sqrt{\mcD_n(u)}\hbox{ for every } u,v\in l(W_n)$ (see \cite[Theorem 1.4.2 (ii)]{FOT}) with $\|u\|_{l^\infty}=\max_{w\in W_n}|u(w)|$. Then, by using $\Gamma$-convergence and equation (\ref{eqn35}), one has $\bar{\mcE}(f\cdot g)\leq 2C_6^{-1}\|f\|^2_{L^\infty(K,\mu)}\bar{\mcE}(g) +2C_6^{-1}\|g\|^2_{L^\infty(K,\mu)}\bar{\mcE}(f)$, hence $f\cdot g\in \bar{\mcF}$.
	
Second, $\bar{\mcF}$ clearly contains constant functions.
	
Finally, we show that for any $x\neq y\in K$, there is $f\in \bar{\mcF}$ such that $f(x)\neq f(y)$. We choose $w\in W_*$ so that $x\in \Psi_wK$ and $y\in \bigcup_{w'\in \mathcal{N}_w^c}\Psi_{w'}K$. For any $n>|w|$, we can find $f_n\in L^2(K,\mu)$ such that $f_n|_{\Psi_wK}=1$, $f_n|_{\bigcup_{w'\in \mathcal{N}_w^c}\Psi_{w'}K}=0$ and $\mcD_n(f_n)\leq R_{n-|w|}^{-1}$. Pick a subsequence $\{f_{n_l}\}_{l\geq 1}$ and $f\in L^2(K,\mu)$ such that $f_{n_l}$ weakly converges to $f$ in $L^2(K,\mu)$. Indeed, $f\in \bar{\mcF}$ since for any $m\geq 1$ we have $r^{-m}\mcD_m(f)=\lim\limits_{l\to\infty}r^{-m}\mcD_m(f_{n_l})\leq 8C_9^{-1}\liminf\limits_{l\to\infty}r^{-n_l}\mcD_{n_l}(f_{n_l}) \leq 8C_9^{-1}\liminf\limits_{l\to\infty}r^{-n_l}R^{-1}_{n_l - |w|}$, where the first inequality follows by a same reason as in the proof of part (a) with $C_9$ being the same constant. In addition, since $f\in \bar{\mcF}\subset C(K)$ (by part (a), which is already proved) and $f_{n_l}$ converges weakly to $f$ in $L^2(K,\mu)$, we have $f(x)=\lim\limits_{m\to\infty}N^m\int_{\Psi_{w^{(m)}}K}f(z)\mu(dz)=\lim\limits_{m\to\infty}\lim\limits_{l\to\infty}N^m\int_{K}1_{\Psi_{w^{(m)}}K}(z)f_{n_l}(z)\mu(dz)=1$, where we denote the indicator function on $A\subset K$ by $1_A$, and we choose $w^{(m)}\in w\cdot W_{m-|w|}$ such that $x\in \Psi_{w^{(m)}}K$ for each $m\geq |w|$. Similarly $f(y)=0$. So $f$ separates $x,y$.\medskip

\textit{$D_4$-symmetry.} In fact, this follows from that $\mcD_n$ is $D_4$-symmetric for each $n\geq 1$. Let us fix $f\in \bar{\mcF}$ and $\Gamma\in \mathscr{G}$. Then, we can see that $f\circ\Gamma\in\bar{\mcF}$ since $\sup_{n\geq1}r^{-n}\mcD_n(f\circ\Gamma)=\sup_{n\geq1}r^{-n}\mcD_n(f)<\infty$. Moreover, by the $\Gamma$-convergence of $r^{-n_l}\mcD_{n_l}$ to $\bar{\mcE}$, we can pick $f_l\in L^2(K,\mu),l\geq 1$ converging to $f$ in $L^2(K,\mu)$ so that $\bar{\mcE}(f)=\lim\limits_{l\to\infty}r^{-n_l}\mcD_{n_l}(f_{l})$, and using the $D_4$-symmetry of $\mcD_{n_l}$, we see
\[
\bar{\mcE}(f\circ \Gamma)\leq\liminf\limits_{l\to\infty}r^{-n_l}\mcD_{n_l}(f_{l}\circ \Gamma)=\lim\limits_{l\to\infty}r^{-n_l}\mcD_{n_l}(f_{l})=\bar{\mcE}(f).
\]
By a same argument, we also see that $\bar{\mcE}(f)\leq\bar{\mcE}(f\circ\Gamma)$. So it follows that $\bar\mcE(f\circ\Gamma)=\bar\mcE(f)$.
\end{proof}

\section{An extension algorithm and its application to \textbf{(B)}}\label{sec4}
The essential difficulty of \textbf{(B)} for $\USC$ is caused by the unconstrained style of neighbourhoods $\mathcal{N}_w$, $w\in W_*$, which may be of infinitely many different patterns. Recall that in \cite{KZ}, the slightly different condition \textbf{(B2)} there is verified by the ``Knight moves'' method due to Barlow and Bass \cite{BB} for a limited range of local symmetric fractals including the classical Sierpinski carpets. However, there are very limited results about slide moves and corner moves in the $\USC$ setting, when the process travels near $\partial_0K$. Also, we can not construct functions of controllable energy simply by only using symmetry as before.

We will take two steps to overcome the difficulty. The argument will be purely analytic.

Firstly, we will provide a lower bound estimate of $R_n(W_{n,2},W_{n,4})$ in terms of $\sigma_n$, i.e. $R_n(W_{n,2}, W_{n,4})\geq C\cdot N^{-n}\sigma_n$ for some $C>0$ independent of $n$. We will take advantage of the strongly recurrent setting in this step. This will be done in Lemma \ref{lemma48}.

Secondly, we will estimate $R_n$ from below by constructing a function $g_w$ for any $w\in W_*\setminus\{\emptyset\}$, which belongs to $l(W_{|w|+n})$, takes value $1$ on $w\cdot W_n$, $0$ on $\mathcal{N}_w^c\cdot W_n$, and has energy controlled by $N^{n}\sigma^{-1}_n$. To fulfill this, first we will develop an extension algorithm by gluing affine copies of functions in $l(W_m)$ with energy controlled by $R_m^{-1}(W_{m,2}, W_{m,4})$, $m\leq n$ (supplied by the first step),  to construct functions in $l(W_n)$ which have nice boundary conditions on $\partial W_n$ and energies controlled by $N^{n}\sigma^{-1}_n$; then again by gluing the affine copies of the obtained functions, and using cutting without increasing energy, we will get the desired function $g_w$ in $l(W_{|w|+n})$. This step will be accomplished in Proposition  \ref{prop46}.

The condition \textbf{(B)} will then follow by combining these two steps.

We start this section with an introduction to the extension algorithm as well as its application to the condition \textbf{(B)}, and leave the estimate of $R_n(W_{n,2},W_{n,4})$ to the last part of this section.


\subsection{Building bricks} \label{sec41}

We will break the fractal $K$ near boundary into infinitely many pieces (building bricks), and glue functions on them together to arrive desired boundary values.

Since all the four boundary lines $L_i,1\leq i\leq 4$ are the same by $\mathscr{G}$-symmetry, we will focus on the bottom line segment $L_1=\overline{q_1,q_2}$. Recall that $W_{m,1}\subset W_{m}$ stands for the collection of $m$-cells that have non-empty intersection with $L_1$, i.e. $W_{m,1}=\{w\in W_m:\Psi_wK\cap L_1\neq \emptyset\}=\{1,2,\cdots, k\}^m$. In addition, the cells $\Psi_iK$ is arranged so that $\Psi_{i-1}K$ is on the left of $\Psi_iK$ for $2\leq i\leq k$.

\begin{definition}\label{def41}
Let $K$ be a $\USC$.
	
(a). For $m\geq 1$, define $T_{m,K}=\bigcup_{w\in W_{m,1}} \Psi_wK$, and call it a \emph{level-$m$ lower wall} of $K$.
	
(b). Let $B_K=\text{cl}\big(T_{1,K}\setminus T_{2,K}\big)$. Here ``cl'' means closure (with respect to the metric $d$).
	
(c). Write $W_{*,1}=\bigcup_{m=0}^\infty W_{m,1}$. For each $w\in W_{*,1}$, call $\Psi_wB_K$ a \emph{level-$|w|$ building brick}.
\end{definition}

We can decompose $T_{1,K}$ into infinitely many building bricks, or into finitely many building bricks together with a lower wall $T_{n,K}$,
\begin{equation}\label{eqn41}
T_{1,K}=\big(\bigcup_{w\in W_{*,1}}\Psi_wB_K\big)\bigcup L_1=\big(\bigcup_{m=0}^{n-2}\bigcup_{w\in W_{m,1}}\Psi_wB_K\big)\bigcup T_{n,K},\quad\text{ for }n\geq 2.
\end{equation}
\begin{figure}[htp]
\includegraphics[width=6cm]{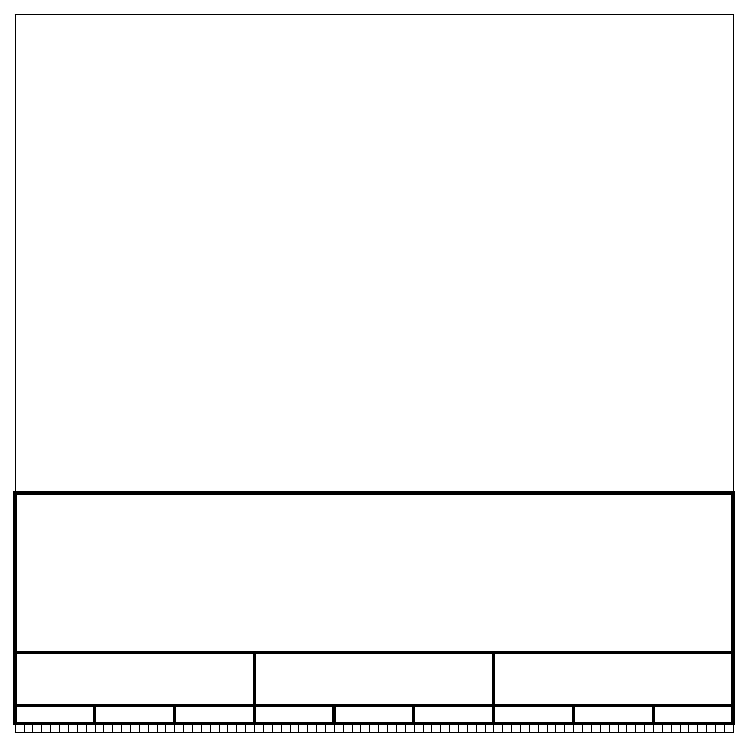}
\begin{picture}(0,0)
\put(-96,35){$B_K$}\put(1,27){$T_{1,K}$}
\end{picture}
\caption{A finite decomposition of $T_{1,K}$ with $n=4$.}
\label{figure2}
\end{figure}
See Figure \ref{figure2} for an illustration. In this section, since we are considering the discrete energies $\mcD_n$ on cell graphs, we focus on the finite decomposition of $T_{1,K}$.

We introduce some new notations for cell graph case.

\begin{definition}\label{def42}
Let $K$ be a $\USC$.
	
(a). For $m\geq 1$ and $n\geq m$, we define $\bar{T}_{m,n}=\{w\in W_n:\Psi_wK\subset T_{m,K}\}$.
	
(b). For $n\geq 2$, we define $\bar{B}_{n}=\{w\in W_n:\Psi_wK\subset B_K\}$.
	
\noindent In particular, $\bar{T}_{n,n}=W_{n,1}$.
\end{definition}

We can rewrite the second decomposition in (\ref{eqn41}) as
\begin{equation}\label{eqn42}
	\bar{T}_{1,n}=\big(\bigsqcup_{m=0}^{n-2}\bigsqcup_{w\in W_{m,1}}w\cdot \bar{B}_{n-m}\big)\bigsqcup \bar{T}_{n,n},\qquad n\geq 2,
\end{equation}
where we use the notation ``$\bigsqcup$'' to emphasize that the unions under consideration are disjoint.

\subsection{Functions with linear boundary values} \label{sec42}

We will use the decomposition (\ref{eqn42}) to construct functions in $l(W_n)$, with good boundary condition on $\partial W_n$, $n\geq 2$. For convenience, in the remaining of this section, for any $w\in W_n,n\geq 1$ and $\Gamma\in \mathscr{G}$, by a little abuse of notation, we let $\Gamma(w)$ denote the unique element in $W_n$ so that  $\Psi_{\Gamma(w)}K=\Gamma(\Psi_wK)$. We will say a function $f\in l(W_n)$ is \emph{$\Gamma$-symmetric} if $f(w)=f(\Gamma (w))$ for all $w\in W_n$.

To start with, let us define a \textit{building brick function} $b_n$ on $\bar{B}_n,n\geq 2$. This takes several steps.\vspace{0.2cm}

\textbf{\textit{Step b.1}}. For $n\geq 1$, let $h_n$ be the unique function on $W_n$ such that
\[\begin{cases}
	&h_n|_{W_{n,4}}=0,\quad h_n|_{W_{n,2}}=1,\\
	&\mcD_n(h_n)=\min\big\{\mcD_n(f):f|_{W_{n,4}}=0, f|_{W_{n,2}}=1, f\in l(W_n)\big\}.
\end{cases}\]
{The existence and uniqueness of $h_n$ is explained in Remark 1 after Definition \ref{def31}. In addition, $h_n$ is $\Gamma_v$-symmetric, i.e. $h_n=h_n\circ \Gamma_v$. In fact, if $h_n\neq h_n\circ \Gamma_v$, then $\mcD_n(\frac{1}{2}h_n+\frac{1}{2}h_n\circ \Gamma_v)= \frac{1}{2}\mcD_n(h_n)+\frac{1}{2}\mcD_n(h_n\circ \Gamma_v)-\mcD_n(\frac{1}{2}h_n-\frac{1}{2}h_n\circ \Gamma_v)<\mcD_n(h_n)$, $(\frac{1}{2}h_n+\frac{1}{2}h_n\circ \Gamma_v)|_{W_{n,4}}=0$ and $(\frac{1}{2}h_n+\frac{1}{2}h_n\circ \Gamma_v)|_{W_{n,2}}=1$, which is a contradiction to the definition of $h_n$. The inequality is due to the facts that $\mcD_n(f)=\mcD_n(f\circ \Gamma_v)\hbox{ for every } f\in l(W_n)$ and $\frac{1}{2}h_n-\frac{1}{2}h_n\circ \Gamma_v$ is not a constant (since $(\frac{1}{2}h_n-\frac{1}{2}h_n\circ \Gamma_v)|_{W_{n,4}}=0$ and $(\frac{1}{2}h_n-\frac{1}{2}h_n\circ \Gamma_v)\neq 0$). Also by the Markov property, we have $0\leq h_n\leq 1$ since $\mcD_n\big((h_n\vee 0)\wedge 1\big)\leq \mcD_n(h_n)$, so $\|h_n\|_{l^{\infty}(W_n)} = 1$.}

For $n\geq 1$, let $A_n=\{w\in W_n:w\sn \bar T_{1,n}\}\cup \bar{T}_{1,n}$, and define $h'_n$ to be the unique function on $W_n$ such that
\[\begin{cases}
&h'_n|_{W_{n,3}}=1,\quad h'_n|_{A_n}=0,\\
&\mcD_n(h'_n)=\min\big\{\mcD_n(f):f|_{W_{n,3}}=1, f|_{A_n}=0, f\in l(W_n)\big\}.
\end{cases}\]
The existence and uniqueness of $h_n'$ is explained in Remark 1 after Definition \ref{def31}. Recall that $\Gamma_h\in \msG$ is defined by $\Gamma_h(x_1,x_2)=(1-x_1,x_2)$ for $x_1,x_2\in \square$, and we also denote $\Gamma_h:W_n\to W_n$ by $\Gamma_h(w)=v$ if $\Gamma_h(\Psi_wK)=\Psi_vK$. By a same argument as before, we can see that $h'_n$ is $\Gamma_h$-symmetric, i.e. $h'_n=h'_n\circ \Gamma_h$. In fact, if $h'_n\neq h'_n\circ \Gamma_h$, then $\mcD_n(\frac{1}{2}h'_n+\frac{1}{2}h'_n\circ \Gamma_h)=\frac{1}{2}\mcD_n(h'_n)+\frac{1}{2}\mcD_n(h'_n\circ \Gamma_h)-\mcD_n(\frac{1}{2}h'_n-\frac{1}{2}h'_n\circ \Gamma_h)<\mcD_n(h'_n)$, $(\frac{1}{2}h'_n+\frac{1}{2}h'_n\circ \Gamma_h)|_{W_{n,3}}=1$ and $(\frac{1}{2}h'_n+\frac{1}{2}h'_n\circ \Gamma_h)|_{A_n}=0$, which is a contradiction to the definition of $h'_n$. Still by the Markov property, $\|h'_n\|_{l^{\infty}(W_n)} = 1$.\vspace{0.2cm}

 We will return to the energy estimate of $h_n$ later in Subsection {\ref{sec43}}. Currently we will estimate energies of other functions in terms of $\mcD_n(h_n)$. In particular, in Step b.1, one has the energy of $h_n'$ well controlled by the energy of $h_{n-1}$ for $n\geq 2$.
\begin{lemma}\label{lemma4add1}
	For $n\geq 2$, $\mcD_n(h'_n)\leq 2k\cdot \mcD_{n-1}(h_{n-1})$.
\end{lemma}
\begin{proof}
We define a function $\tilde{h}_n'$ in $l(W_n)$ as follows:
\[\tilde{h}_{n}'(i\cdot w)=
\begin{cases}
h_{n-1}\circ \Gamma_{r_3}(w),&\text{ if }i\in W_{1,3},w\in W_{n-1},\\
0,&\text{ if }i\notin W_{1,3},w\in W_{n-1}.\\	
\end{cases}\]
Then $\tilde{h}_{n}'|_{W_{n,3}}=1$ and $\tilde{h}_n'|_{A_n}=0$. In addition, by the $\Gamma_v$-symmetry of $h_{n-1}$ (see the discussion in Step b.1), one can check directly that $\tilde{h}_n'(i\cdot w)=\tilde{h}_n'\big(j\cdot\Gamma_h(w)\big)$ for $i,j\in W_{1,3}$ with $i-j=1$ and $w\in W_{n-1,2}$. As a consequence, the cross energy $\mcD_{n,i\cdot W_{n-1,2}\cup j\cdot W_{n-1,4}}(\tilde{h}_n')- \mcD_{n,i\cdot W_{n-1,2}}(\tilde{h}_n') -\mcD_{n,j\cdot W_{n-1,4}}(\tilde{h}_n')$ between $i\cdot W_{n-1,2}$ and $j\cdot W_{n-1,4}$ satisfies
\[
\begin{aligned}
&\mcD_{n,i\cdot W_{n-1,2}\cup j\cdot W_{n-1,4}}(\tilde{h}_n')- \mcD_{n,i\cdot W_{n-1,2}}(\tilde{h}_n') -\mcD_{n,j\cdot W_{n-1,4}}(\tilde{h}_n')\\
=&\mcD_{n,i\cdot W_{n-1,2}}(\tilde{h}_n')+\mcD_{n,j\cdot W_{n-1,4}}(\tilde{h}_n') \leq \mcD_{n-1}(h_{n-1}),
\end{aligned}
\]
and thus
\[\begin{aligned}
\mcD_n(h'_n)\leq \mcD_n(\tilde{h}'_n)=&\sum_{i\in W_{1,3}}\mcD_{n,i\cdot W_{n-1}}(\tilde{h}_n')\\&+\sum_{i,j\in W_{1,3}\atop i-j=1}\big(\mcD_{n,i\cdot W_{n-1,2}\cup j\cdot W_{n-1,4}}(\tilde{h}_n')- \mcD_{n,i\cdot W_{n-1,2}}(\tilde{h}_n') -\mcD_{n,j\cdot W_{n-1,4}}(\tilde{h}_n')\big)\\
\leq &2k\cdot\mcD_{n-1}(h_{n-1}).
\end{aligned}\]
\end{proof}\vspace{0.2cm}

\textbf{\textit{Step b.2}}. Next, for $n\geq 2$, we define $b_n^{(1)}\in l(\bar{B}_n)$ and $b_n^{(2)}\in l(\bar{B}_n)$ as
\[b_n^{(1)}=h_n|_{\bar{B}_n},\]
and
\[b_{n}^{(2)}(i\cdot w)=\frac{i-1}{k}+\frac{1}{k}h_{n-1}(w)\quad \hbox{ for every } 1\leq i\leq k,  w\in W_{n-1}\setminus \bar{T}_{1,n-1}.\]

\begin{lemma}\label{lemma4add2}
(a). For $n\geq 2$, we have $b_n^{(1)}|_{W_{n,4}\cap\bar B_n}\equiv 0$ and $b_n^{(1)}|_{W_{n,2}\cap\bar B_n}\equiv 1$.

(b). For $n\geq 2$ and $1\leq i\leq k$, we have
\[b_n^{(2)}|_{(i\cdot W_{n-1,4})\cap \bar{B}_n}\equiv\frac{i-1}{k},\quad b_n^{(2)}|_{(i\cdot W_{n-1,2})\cap \bar{B}_n}\equiv\frac{i}{k}.\]

(c). For $n\geq 2$, we have $\mcD_{n,\bar{B}_n}(b_n^{(1)})\leq \mcD_{n}(h_n)$ and $\mcD_{n,\bar{B}_n}(b_n^{(2)})\leq \frac{1}{k}\cdot\mcD_{n-1}(h_{n-1})$.
\end{lemma}
\begin{proof}
It is straightforward to check (a) and (b), noticing that $h_{n}|_{W_{n},4}=0$ and $h_{n}|_{W_{n},2}=1$.

(c). There is nothing to say about $\mcD_{n,\bar{B}_n}(b_n^{(1)})$. For $b_n^{(2)}$, by using (b), we can see that $b_n^{(2)}(v)=b_n^{(2)}(v')$ for all $v,v'\in \bar{B}_n$ such that $v\sn v'$ and  $v_1\neq v'_1$ ($v_1,v'_1$ are the leading symbols of $v,v'$ respectively), and thus
\[\mcD_{n,\bar{B}_n}(b_n^{(2)})=\sum_{i=1}^k\mcD_{n-1,W_{n-1}\setminus \bar{T}_{1,n-1}}(\frac{i-1}{k}+\frac{1}{k}h_{n-1})\leq \frac{1}{k}\cdot \mcD_{n-1}(h_{n-1}).\]
\end{proof}

\textbf{\textit{Step b.3}}. Lastly, for $n\geq 2$, we define $b_n\in l(\bar B_n)$ as
\[b_n=b_n^{(1)}\cdot h''_n+b^{(2)}_n\cdot (1-h''_n),\]
where $h''_n\in l(\bar B_n)$ is defined as
\[h''_n(i\cdot w)=h'_{n-1}(w)\quad\hbox{ for every } 1\leq i\leq k, w\in W_{n-1}\setminus\bar T_{1,n-1}.\]

We include the information about boundary values (including left boundary $W_{n,4}\cap \bar B_n$, right boundary $W_{n,2}\cap \bar B_n$, upper boundary $W_{1,1}\cdot W_{n-1,3}$ and lower boundary $W_{1,1}\cdot (A_{n-1}\setminus \bar T_{1,n-1})$) and energy estimate of the building brick function $b_n\in l(\bar B_n)$ in the following Lemma \ref{lemma43}.

\begin{lemma}\label{lemma43}
Let $n\geq 2$ and $b_n\in l(\bar{B}_n)$ be the function defined through Step b.1 to b.3.
	
(a). $b_n|_{W_{n,4}\cap \bar B_n}=0$,  $b_n|_{W_{n,2}\cap \bar B_n}=1$.
	
(b). $b_n(w)=h_n(w)\hbox{ for } w\in W_{1,1}\cdot W_{n-1,3}$.
	
(c). $b_n(i\cdot w)=\frac{i-1}{k}+\frac{1}{k}h_{n-1}(w)\hbox{ for } 1\leq i\leq k, \, w\in A_{n-1}\setminus \bar T_{1,n-1}$.
	
(d). $\mcD_{n,\bar{B}_n}(b_n)\leq C\cdot  \big(\mcD_{n}(h_{n})+\mcD_{n-1}(h_{n-1})+\mcD_{n-1}(h'_{n-1})\big)$, where $C>0$ is independent of $n$.
\end{lemma}
\begin{proof}
(a). One can see that $b^{(i)}_n|_{W_{n,4}\cap \bar B_n}=0$ and $b^{(i)}_n|_{W_{n,2}\cap \bar B_n}=1$ for $i=1,2$ from Lemma \ref{lemma4add2} (a), (b). Thus (a) follows immediately.
	
(b) and (c). By the definition of $h''_n$, one can see
\[h''_n|_{W_{1,1}\cdot W_{n-1,3}}=1,\quad h''_n|_{W_{1,1}\cdot (A_{n-1}\setminus \bar T_{1,n-1})}=0.\]
So we have
\[b_n|_{W_{1,1}\cdot W_{n-1,3}}=b_n^{(1)}|_{W_{1,1}\cdot W_{n-1,3}},\quad b_n|_{W_{1,1}\cdot (A_{n-1}\setminus \bar T_{1,n-1})}=b_n^{(2)}|_{W_{1,1}\cdot (A_{n-1}\setminus \bar T_{1,n-1})}.\]
Then (b) follows immediately from the definition of $b_n^{(1)}$, and (c) follows immediately from the definition of $b_n^{(2)}$.

(d). First, by the $\Gamma_h$-symmetry of $h'_{n-1}$ (see the discussion in Step b.1), using the same argument as that in the proof of Lemma \ref{lemma4add1}, we have $\mcD_{n,\bar B_n}(h''_n)\leq 2k\mcD_{n-1}(h'_{n-1})$. Then, noticing that $\|b_n^{(1)}\|_{l^\infty(\bar{B}_n)}=1,\|b_n^{(2)}\|_{l^\infty(\bar{B}_n)}=1$ and $\|h''_n\|_{l^\infty(\bar{B}_n)}=1$, we have (d) holds by using Lemma \ref{lemma4add2} (c) and the well known inequality (see \cite[Theorem 1.4.2]{FOT}) $\sqrt{\mcD_{n,\bar{B}_n}(f\cdot g)}\leq \|f\|_{l^\infty(\bar{B}_n)}\sqrt{\mcD_{n,\bar B_n}(g)}+\|g\|_{l^\infty(\bar{B}_n)}\sqrt{\mcD_{n, \bar B_n}(f)}$ for every $f,g\in l(\bar{B}_n)$.
\end{proof}\vspace{0.2cm}

Next, we use the building brick functions to build a function $f_n\in l(W_n)$, whose restriction to the boundary is  ``linear''. In addition, We will show that $\mcD_n(f_n)$ has an upper bound estimate in terms of $\mcD_m(h_m),1\leq m\leq n$.

\begin{lemma}\label{lemma44}
Let $u\in l(K)$ defined by $u(x_1,x_2)=x_1\hbox{ for every } x=(x_1,x_2)\in K$. Then there is $f_n\in l(W_n)$ such that $f_n|_{\partial W_n}=P_n u|_{\partial W_n}$ and
\begin{equation}\label{eqn43}
\mcD_n(f_n)\leq C\cdot\big(k^{-n}+\sum_{m=0}^{n-1} k^{-m}\mcD_{n-m}(h_{n-m})\big),
\end{equation}
for some constant $C>0$ independent of $n\geq 2$.
\end{lemma}
\begin{proof}
For each $m\geq 1$, we define $e_m(w)=\sum_{j=1}^m k^{m-j}(w_j-1)$ for $w=w_1\cdots w_m\in W_{m,1}$. In particular, we write $e_0(\emptyset) = 0$. For convenience, we first consider $u'_n\in l(K)$ defined by $u'_n(x_1,x_2)=\frac{k^n}{k^n-1}(x_1-\frac{1}{2k^n})$ instead of $u$. Then we have $(P_n u'_n)(w)=\frac{e_n(w)}{k^n-1}$. In particular, $(P_n u'_n)(1\cdots 1)=0$ and $(P_n u'_n)(k\cdots k)=1$. Now we construct a function $g_n\in l(W_n)$ such that $g_n|_{\partial W_n}=P_nu'_n|_{\partial W_n}$ by the following two steps.\vspace{0.2cm}
	
\noindent\textit{Step 1. We first define $g'_n$ on $\bar T_{1,n}$.}
	
We denote the indicator function of a set $A$ by $1_A$. Define
\[g'_n=(P_n u_n')\cdot 1_{\bar T_{n,n}}+\sum_{m=0}^{n-2}\sum_{w\in W_{m,1}}\frac{1}{k^m}\big(e_m(w)+b_{n-m}\circ \Psi_{w}^{-1}\big)\cdot 1_{w\cdot \bar{B}_{n-m}},\]
where  by a little abuse of the notation, $b_{n-m}\circ \Psi_w^{-1}$ is the function supported on $w\cdot \bar{B}_{n-m}$ such that $b_{n-m}\circ \Psi_w^{-1}(w\cdot \eta)=b_{n-m}(\eta)$ for any $\eta\in \bar{B}_{n-m}$.\vspace{0.2cm}
	
\noindent\textit{Step 2. We use symmetry to extend $g_n'$ to $g_n\in l(W_n)$.}
	
We define
\[g_n(w)=
\begin{cases}
	g'_n(w),&\text{ if }w\in \bar{T}_{1,n},\\
	g'_n\circ \Gamma_v(w),&\text{ if }w\in \Gamma_v\bar{T}_{1,n},\\
	h_n(w),&\text{ otherwise}.	
\end{cases}\]
Clearly, by Lemma \ref{lemma43} (a), $g_n$ satisfies the boundary condition $g_n|_{\partial W_n}=P_n u_n'|_{\partial W_n}$. In addition, $g_n$ is $\Gamma_v$-symmetric, since $h_n$ is $\Gamma_v$-symmetric by the discussion in Step b.1.\vspace{0.2cm}
	
In the following, we estimate $\mcD_{n, \bar T_{1,n}}(g_n)$.
	
Firstly, we consider the energy of $g_n$ on each ``layer'' of building bricks: for $0\leq m\leq n-2$, we have
\begin{equation}\label{eqn4add1}
\begin{aligned}
\mcD_{n,\bar T_{m+1,n}\setminus \bar T_{m+2,n}}(g_n)&=\sum_{w\in W_{m,1}}\mcD_{n,w\cdot \bar{B}_{n-m}}\big(\frac{1}{k^m}e_m(w)+\frac{1}{k^m}b_{n-m}\circ \Psi_{w}^{-1}\big)\\
&=\frac1{k^m}\mcD_{n-m,\bar{B}_{n-m}}(b_{n-m}),
\end{aligned}
\end{equation}
where the first equality is because $g_n(v)=g_n(v')$ for any $v,v'\in \bar T_{m+1,n}\setminus \bar T_{m+2,n}$ such that $v\sn v'$ and $v_1v_2\cdots v_{m}\neq v'_1v'_2\cdots v'_{m}$, as a consequence of Lemma \ref{lemma43} (a) and the construction of $g_n$.

Secondly, there is $C_1>0$ independent of $n$ such that
\begin{equation}\label{eqn4add2}
\mcD_{n,\bar T_{n-1,n}}(g_n)\leq C_1\cdot k^{-n},
\end{equation}
since there are at most $4Nk^{n-1}$ pairs of $n$-cells $v\sn v'$ in $\bar T_{n-1,n}$ where $4$ appears since each $n$-cell is neighbouring to at most $8$ $n$-cells and $Nk^{n-1}$ is the number of $n$-cells in $\bar T_{n-1,n}$, and $|g_n(v)-g_n(v')|\leq k^{-n+2}$ by the construction of $g_n$ where $k^{-n+2}$ is the oscillation of $g_n$ in $w\cdot \bar{T}_{1,2}$ for each $w\in W_{n-2,1}$ (we also use the observation that $|g_n(v)-g_n(v')|\leq k^{-n}$ for any $w\neq w'\in W_{n-2,1}$, $v\in w\cdot \bar{T}_{1,2},v'\in w'\cdot \bar{T}_{1,2}$ such that $v\sn v'$).

Thirdly, we consider the  energy of $g_n$ between two ``layers'' of building bricks: for $0\leq m\leq n-1$, let
\[E_{m,n}=\{w\in \bar{T}_{m+1,n}:w\sn \bar{T}_{m,n}\setminus \bar{T}_{m+1,n}\}\bigcup \{w\in \bar{T}_{m,n}\setminus \bar{T}_{m+1,n}:w\sn \bar{T}_{m+1,n}\},\]
where we write $\bar T_{0,n}:=W_n$. In particular, we are interested in $E_{m,n}$, $1\leq m\leq n-2$. For $1\leq m\leq n-2$, one can check that $E_{m,n}=\bigcup_{w\in W_{m,1}}w\cdot E_{0,n-m}$, and in addition, by Lemma \ref{lemma43} (b), (c), and the construction of $g_n$, one can see that
\[g_n(w\cdot v)=\frac{e_m(w)}{k^{m}}+\frac{1}{k^{m}}h_{n-m}(v)\quad \hbox{ for every } w\in W_{m,1}, v\in E_{0,n-m}.\]
So we have for $1\leq m\leq n-2$,
\begin{equation}\label{eqn4add3}
\mcD_{n,E_{m,n}}(g_n)=\sum_{w\in W_{m,1}}\mcD_{n,w\cdot E_{0,n-m}}\big(\frac{e_m(w)+h_{n-m}\circ \Psi_w^{-1}}{k^{m}}\big)\leq \frac{\mcD_{n-m}(h_{n-m})}{k^{m}},
\end{equation}
where the equality holds due to a same reason as the first case.


Finally, adding up (\ref{eqn4add1}),(\ref{eqn4add2}) and (\ref{eqn4add3}) together,  using Lemma \ref{lemma4add1} and Lemma \ref{lemma43} (d), and viewing $\mcD_1(h'_1)$ as a constant depending only on $K$, we get the energy estimate
\[\mcD_{n,\bar T_{1,n}}(g_n)\leq C_2\cdot\big(k^{-n}+\sum_{m=0}^{n-1} k^{-m}\mcD_{n-m}(h_{n-m})\big),\]
where $C_2>0$ is a constant independent of $n$.

Since $g_n$ is $\Gamma_v$-symmetric and $g_n|_{W_{1,1}\cdot W_{n-1,3}}=h_n|_{W_{1,1}\cdot W_{n-1,3}}$, we get $\mcD_n(g_n)\leq C_3\cdot \big(\mcD_{n,\bar T_{1,n}}(g_n)+\mcD_n(h_n)\big)$ for some $C_3>0$.
    	
At last, we take $f_n=\frac{k^n-1}{k^n}g_n+\frac{1}{2k^n}$ so that $f_n|_{\partial W_n}=P_nu|_{\partial W_n}$. The energy estimate (\ref{eqn43}) of $f_n$ follows immediately from that of $g_n$.
\end{proof}

In the next subsection, we will show the following two estimates.
\begin{eqnarray}\label{eqn44}
C\cdot N^{-n+m}\sigma_{n-m}\geq (\frac{2}{k})^mN^{-n}\sigma_{n}\quad &\hbox{ for every } n\geq 1 ,0\leq m\leq n-1,\\
\label{eqn45}
\mcD_n(h_n)\leq C\cdot N^n\sigma_n^{-1}\quad\quad&\hbox{ for every } n\geq 1,
\end{eqnarray}
for some $C>0$ independent of $n$. With these estimates, we have a clearer explanation of (\ref{eqn43}).

\begin{corollary}\label{coro45}
	Assume (\ref{eqn44}) and (\ref{eqn45}) hold. Let $f_n$  be the same function in Lemma \ref{lemma44}. Then
	\begin{equation}\label{eqn46}
		\mcD_n(f_n)\leq C\cdot N^n\sigma_n^{-1}
	\end{equation}
	for some constant $C>0$ independent of $n\geq 2$.
\end{corollary}
\begin{proof}
By (\ref{eqn44}) and (\ref{eqn45}), there is a constant $C_1>0$ independent of $n$ so that
\[
\mcD_{n-m}(h_{n-m})\leq C_1\cdot (\frac{k}{2})^m N^n\sigma_n^{-1}\quad\hbox{ for every } n\geq 1,0\leq m\leq n-1,
\]
and
\begin{equation}\label{eqn4add4}
C_1\cdot N^n\sigma_n^{-1}\geq (\frac{2}{k})^n\quad\hbox{ for every } n\geq 1.
\end{equation}
Substituting the above estimates into (\ref{eqn43}), we see that
\[
\mcD_n(f_n)\leq C_2\cdot\big(k^{-n}+\sum_{m=0}^{n-1} k^{-m}\mcD_{n-m}(h_{n-m})\big)\leq C_1C_2\cdot\sum_{m=0}^n\frac{1}{2^m}N^n\sigma_n^{-1}\leq 2C_1C_2\cdot N^n\sigma_n^{-1},
\]
where $C_2$ is the constant $C$ in Lemma \ref{lemma44}.
\end{proof}

Lemma \ref{lemma44} and Corollary \ref{coro45} are enough for us to verify \textbf{(B)}.

\begin{proposition}\label{prop46}
Assume  (\ref{eqn44}) and (\ref{eqn45}) hold. Then \textbf{(B)} holds.
\end{proposition}
\begin{proof}
Let $n,m\geq 1$ and $w\in W_m$. We need to estimate $R_{m+n}\big(w\cdot W_n,\mathcal{N}_w^c\cdot W_n\big)$.
	
Let $c_0$ be the constant in \textbf{(A3)}, and let $p=(p_1,p_2)$ be the center of $\Psi_wK$, i.e. $p=\Psi_w(z)$ with $z=(\frac12,\frac12)$. We consider the following four linear functions on $K$,
\[\begin{aligned}
u_1(x_1,x_2)=k^mc_1(x_1-p_1)+c_2,\quad u_2(x_1,x_2)=k^mc_1(p_1-x_1)+c_2,\\
u_3(x_1,x_2)=k^mc_1(x_2-p_2)+c_2,\quad u_4(x_1,x_2)=k^mc_1(p_2-x_2)+c_2,
\end{aligned}\]
for $x=(x_1,x_2)\in K$, where $c_1=\sqrt{2}c_0^{-1}$ and $c_2=1+\frac 12 c_1$. In particular, let $u'=\min_{i=1}^4P_{m+n}(u_i)$, one can check that $u'|_{w\cdot W_n}\geq 1$ and $u'|_{\mathcal{N}_w^c\cdot W_n}\leq 0$.
	
By Lemma \ref{lemma44} and Corollary \ref{coro45}, we can construct $g_i\in l(W_{m+n})$, $1\leq i\leq 4$ so that
$g_i(w'\cdot \tau)=P_n(u_i\circ \Psi_{w'})(\tau)\hbox{ for every }w'\in W_m,\tau\in \partial W_n$ and $\mcD_{m+n,w'\cdot W_n}(g_i)\leq C\cdot c_1^2N^n\sigma_n^{-1}$, where $C$ is the same constant in Corollary \ref{coro45}. Indeed, for the case $i=1$, for each $w'\in W_m$, by the definition of $u_1$, we see that $u_1(x_1,x_2)=k^mc_1(x_1-p_1^{w'})+c_{w'}$ on $\Psi_{w'}K$ for some constant $c_{w'}\in \mathbb{R}$, where we denote the center of $\Psi_{w'}K$ as $p^{w'}=(p_1^{w'}, p_2^{w'})$. So $g_1(w'\cdot \tau)=c_1f_n(\tau)+c_{w'}-\frac{c_1}{2}$ for each $w'\in W_m$, $\tau\in W_n$, where $f_n$ is the same function in Lemma \ref{lemma44}. The cases $i=2,3,4$ are the same by using symmetry.
	
In addition, due to the construction of $g_i$, for each pair $w'\sm w''$ in $W_m$, and $\tau', \tau''\in \partial W_n$ such that $w'\cdot \tau'\stackrel{m+n}{\sim} w''\cdot \tau''$, the difference $|g_i(w'\cdot \tau')-g_i(w''\cdot \tau'')|$ is controlled by $c_1k^{-n}$. Since there are at most $3k^n$ pairs of such $\tau'$ and $\tau''$ for $w'\sm w''$, this gives that
\[\mcD_{m+n, A}(g_i)\leq C'\cdot N^{n}\sigma^{-1}_n+C'\cdot k^{-n}\leq C''\cdot N^n\sigma^{-1}_n,\]
where $A=(\mathcal N_w\cdot W_n)\bigcup\{v\in W_{m+n}:v\stackrel{m+n}{\sim} \mathcal N_w\cdot W_n\}$, $C',C''>0$ are independent of $n,m,w$, and the second inequality is a consequence of (\ref{eqn4add4}).
	
Finally, let $g=\big((\min_{i=1}^4 g_i)\vee 0\big)\wedge 1$. One can see that $g|_{w\cdot W_n}=1$, $g|_{\mathcal{N}_w^c\cdot W_n}=0$ and $\mcD_{m+n}(g)=\mcD_{m+n,A}(g)\leq 4C''\cdot N^n\sigma^{-1}_n$. The condition \textbf{(B)} follows immediately.
\end{proof}

\subsection{Verification of the estimates} \label{sec43}
We prove formulas (\ref{eqn44}) and (\ref{eqn45}) in this last subsection. The proof is essentially based on $\lambda_n\leq C\cdot \sigma_n$ by (\ref{eqn33}) and Lemma \ref{lemma35}.

First, we prove (\ref{eqn45}), which is equivalent to the effective resistance estimate
\[R_n(W_{n,2},W_{n,4})\geq C\cdot N^{-n}\sigma_n\]
for some $C>0$ independent of $n$.

Let $n\geq 1, m\geq 1$. For a connected subset $A$ of $W_{m,1}$ with $l=\#A\geq 2$, let $w_A^{(1)}$ be the most left one  and $w_A^{(2)}$ be the most right one in $A$, viewing $A$ as a chain of cells. Define
\[\tilde{R}_{n,m,A}=\max\big\{\mcD_{m+n,A\cdot W_n}^{-1}(f):f\in l(A\cdot W_n),f|_{w_A^{(1)}\cdot W_n}=0,f|_{w_A^{(2)}\cdot W_n}=1
\big\}.\]
Clearly, $\tilde{R}_{n,m,A}$ depends only the length of $A$, i.e. $l=\#A$, and is independent of $m$. Denote $\tilde{R}_{n,m,A}$ as  $\tilde{R}_{n,l}$. Since we are working on finite graphs, $\tilde{R}_{n,l}$ are positive and finite for any $n\geq 1$ and $l\geq 2$.

\begin{lemma}\label{lemma47}
For $n\geq 1$, we have $R_n(W_{n,2},W_{n,4})\geq\frac{1}{4l}\cdot\tilde{R}_{n,l}$.
\end{lemma}
\begin{proof}
	One can see that $\tilde{R}_{n,l}$ is increasing in $l$, i.e. $\tilde{R}_{n,l_1}\leq \tilde{R}_{n,l_2}$ if $l_1\leq l_2$. To see this, we choose large enough $m$ and $A_1\subset A_2\subset W_{m,1}$ such that $\#A_1=l_1$, $\#A_2=l_2$, and $111\cdots 1\in A_1$. Let $h_1\in l(A_1\cdot W_n)$ so that $h_1|_{w_{A_1}^{(1)}\cdot W_n}=0$, $h_1|_{w_{A_1}^{(2)}\cdot W_n}=1$ and $\mcD_{m+n,A_1\cdot W_n}(h_1)=\tilde{R}^{-1}_{n,l_1}$. By  extending $h_1$ to $h_2\in l(A_2\cdot W_n)$ with
\[h_2|_{A_1\cdot W_n}=h_1,\quad h_2|_{(A_2\setminus A_1)\cdot W_n}=1,\]
we immediately have $h_2|_{w^{(1)}_{A_2}\cdot W_n}=0$ and $h_2|_{w^{(2)}_{A_2}\cdot W_n}=1$, and so
\[\tilde{R}_{n,l_2}^{-1}\leq\mcD_{m+n,A_2\cdot W_n}(h_2)=\mcD_{m+n,A_1\cdot W_n}(h_1)=\tilde{R}_{n,l_1}^{-1}.\]
	
For the case $l=2$, i.e. $A = \{w^{(1)}_A,w^{(2)}_A\}$ with $w^{(1)}_A\stackrel{m}{\sim}w^{(2)}_A$, the lemma is easy since we have $R_n(W_{n,2}, W_{n,4})^{-1} \leq \mathcal{D}_{n}(1_{W_{n,2}}) = \#\{(v,v'):v\stackrel{n}{\sim}v', v\in W_{n,2},v'\in W_n\setminus W_{n,2}\} \leq \#\{(\tau,\tau'):\tau\stackrel{m+n}{\sim}\tau', \tau\in w_{A}^{(1)}\cdot W_{n,2},\tau'\in w_{A}^{(2)}\cdot W_{n,4}\} = \tilde{R}_{n,2}^{-1}$, where $1_{B}$ denotes the indicator function of $B$ in $W_n$.

Then we look at the case $l = 3$, i.e. ${A=\{w^{(1)}_A,w,w^{(2)}_A\}}$ such that ${w^{(1)}_A\sm w\sm w^{(2)}_A}$. Choose $f\in l(A\cdot W_n)$ so that {$f|_{w^{(1)}_A\cdot W_n}=0,f|_{w^{(2)}_A\cdot W_n}=1$} and $\mcD_{m+n,A\cdot W_n}(f)=\tilde{R}_{n,3}^{-1}$. Define $g\in l(W_n)$ as
\[g(v)=1_{W_{n,2}}(v)+f(w\cdot v)\cdot 1_{W_n\setminus (W_{n,2}\cup W_{n,4})}(v)+0\cdot 1_{W_{n,4}}(v)\quad\hbox{ for every } v\in W_n.\]
Then, we have $g|_{W_{n,4}}=0$, $g|_{W_{n,2}}=1$. In addition,
\begin{align*}
&\sum_{v\in W_{n,4}}\sum_{v':v'\sn v}\big(g(v)-g(v')\big)^2
=\sum_{v\in W_{n,4}}\sum_{v'\in W_n\setminus W_{n,4}:v'\sn v}\big(0-f(w\cdot v')\big)^2\\
\leq& 2\sum_{v\in W_{n,4}}\sum_{v'\in W_n\setminus W_{n,4}:v'\sn v} \Big(\big(0-f(w\cdot v)\big)^2+\big(f(w\cdot v)-f(w\cdot v')\big)^2\Big)\\
\leq& 2\sum_{v\in W_{n,4}}\Big(\sum_{\tau'\in w_A^{(1)}\cdot W_n:\tau'\stackrel{m+n}{\sim}w\cdot v}\big(f(\tau')-f(w\cdot v)\big)^2+\sum_{v'\in W_n\setminus W_{n,4}:v'\sn v} \big(f(w\cdot v)-f(w\cdot v')\big)^2\Big)\\
\leq&2\sum_{\tau\in w\cdot W_{n,4}}\sum_{\tau'\in A\cdot W_n:\tau'\stackrel{m+n}{\sim}\tau}\big(f(\tau)-f(\tau')\big)^2,
\end{align*}
where the equality holds as $g(v)=g(v')=0$ for each $v,v'\in W_{n,4}$, and the second inequality holds because $f|_{w_A^{(1)}\cdot W_n}=0$ and for each $v\in W_{n,4}$, we have $\#\{\tau'\in w_A^{(1)}\cdot W_n:\tau'\stackrel{m+n}{\sim} w\cdot v\}\geq \#\{v'\in W_n\setminus W_{n,4}:v'\sn v\}$. One can apply the same argument to $W_{n,2}$ by $\Gamma_h$-symmetry.  Thus,
\[\begin{aligned}
\mcD_{n}(g)&=\sum_{v\in W_{n,2}\bigcup W_{n,4}}\sum_{v':v'\sn v}\big(g(v)-g(v')\big)^2+\mcD_{n,W_{n}\setminus (W_{n,4}\cup W_{n,2})}(g)\\
&\leq 2\sum_{i=2,4}\sum_{\tau\in w\cdot W_{n,i}}\sum_{\tau'\in A\cdot W_n:\tau'\stackrel{m+n}{\sim}\tau}\big(f(\tau)-f(\tau')\big)^2+\mcD_{m+n,w\cdot\big(W_{n}\setminus (W_{n,4}\cup W_{n,2})\big)}(f)\\
&\leq 2\cdot \mcD_{m+n,A\cdot W_n}(f).
\end{aligned}\]
Thus,
\[{R_n(W_{n,2},W_{n,4})}\geq \big(\mathcal{D}_{n}(g)\big)^{-1}\geq \frac{1}{2}\mcD^{-1}_{m+n,A\cdot W_n}(f)= \frac{1}{2}\tilde{R}_{n,3}.\]
	
Next, we consider the cases $l=2^s+2$, $s\geq 0$. By applying symmetry, one can see
\begin{equation}\label{iteration}
\tilde{R}_{n,2^{s+1}+2}\leq 2\cdot \tilde{R}_{n,2^s+2}.
\end{equation}
To see this, we fix a connected $A\subset W_{m,1}$ for some large enough $m\geq1$ with $\#A=2^{s+1}+2$, and choose $f\in l(A\cdot W_n)$ so that $f|_{w^{(1)}_A\cdot W_n}=0$, $f|_{w^{(2)}_A\cdot W_n}=1$ and $\mcD_{m+n,A\cdot W_n}(f)=\tilde{R}_{n,2^{s+1}+2}^{-1}$. To apply symmetry, we divide $A$ in the middle: let $A=A_1\bigsqcup A_2$, where $A_1,A_2$ are connected and $\#A_1=\#A_2=2^s+1$, $w_A^{(1)}\in A_1,w_A^{(2)}\in A_2$. Then, $f-\frac{1}{2}$ is antisymmetric with respect to the reflection $\Gamma_A$ which interchanges $A_1,A_2$ by a routine argument. In fact, if $f-\frac12$ is not $\Gamma_A$-antisymmetric, which means $f-\frac12\neq -f\circ\Gamma_A+\frac12$, then $\frac{f-f\circ \Gamma_A+1}{2}|_{w^{(1)}_A\cdot W_n}=0$, $\frac{f-f\circ \Gamma_A+1}{2}|_{w^{(2)}_A\cdot W_n}=1$ and $\mcD_{m+n, A\cdot W_n}(\frac{f-f\circ \Gamma_A+1}{2})<\frac12\mcD_{m+n,A\cdot W_n}(f)+\frac12\mcD_{m+n,A\cdot W_n}(f\circ \Gamma_A)=\mcD_{m+n,A\cdot W_n}(f)$, which leads to a contradiction as $f$ minimizes $\mcD_{m+n,A\cdot W_n}(\cdot )$ among functions in $l(A\cdot W_n)$ that are $0$ on $w^{(1)}_A\cdot W_n$ and $1$ on $w^{(2)}_A\cdot W_n$. Using the $\Gamma_A$-antisymmetry of $f-\frac 12$, we can show
\begin{equation}\label{eqn4add5}
f|_{A_1\cdot W_n}\leq \frac{1}{2},\quad f|_{A_2\cdot W_n}\geq \frac{1}{2}.
\end{equation}
We prove \eqref{eqn4add5} by contradiction. For short, we denote $A_c=\{v\in A_1\cdot W_n:v\stackrel{m+n}\sim A_2\cdot W_n\}\cup \{v\in A_2\cdot W_n:v\stackrel{m+n}\sim A_1\cdot W_n\}$, where $c$ means ``cross''. Assume \eqref{eqn4add5} does not hold, we can define $g=(f\wedge \frac{1}{2})\cdot 1_{A_1\cdot W_n}+(f\vee \frac{1}{2})\cdot 1_{A_2\cdot W_n}$, and see that
\[\begin{aligned}
\mcD_{m+n,A\cdot W_n}(g)
=&\sum_{v\stackrel{m+n}\sim v':\{v,v'\}\subset A\cdot W_n,\atop \{v,v'\}\not\subset A_c}\big(g(v)-g(v')\big)^2
+\sum_{v\in (A_1\cdot W_n)\cap A_c}\big(g(v)-g(\Gamma_A v)\big)^2\\
&+\sum_{v\stackrel{m+n}\sim v':\{v,v'\}\subset (A_1\cdot W_n)\cap A_c}\sum_{l,l'=0,1}\big(g(\Gamma_A^{l'}v)-g(\Gamma_A^lv')\big)^2\\\leq& \mcD_{m+n,A\cdot W_n}(f),
\end{aligned}\]
where the first and third summations are about unordered pairs $\{v,v'\}$, for terms in the first summation, $|f(v)-f(v')|\geq |g(v)-g(v')|$ since either $\{v,v'\}\subset A_1\cdot W_n$ or $\{v,v'\}\subset A_2\cdot W_n$; for the second summation, $|f(v)-f(\Gamma_Av)|\geq |g(v)-g(\Gamma_A v)|$ since $f-\frac{1}{2}$ is $\Gamma_A$-antisymmetric; for the last summation, it always holds $\sum_{l=0,1}|f(v)-f(\Gamma_A^{l}v')|^2\geq \sum_{l=0,1}|g(v)-f(\Gamma_A^{l}v')|^2\geq \sum_{l=0,1}|g(v)-g(\Gamma_A^{l}v')|^2$ and $\sum_{l=0,1}|f(\Gamma_Av)-f(\Gamma_A^{l}v')|^2\geq \sum_{l=0,1}|g(\Gamma_Av)-g(\Gamma_A^{l}v')|^2$ still since $f-\frac{1}{2}$ is $\Gamma_A$-antisymmetric. This gives (\ref{eqn4add5}) since $f$ is unique. By (\ref{eqn4add5}), we have $\mcD_{m+n,A\cdot W_n}(f\vee \frac{1}{2})=\mcD_{m+n,A\cdot W_n}(f\wedge \frac{1}{2})\geq\frac{1}{4}\tilde R_{n,2^s+2}^{-1}$, whence
\[\tilde{R}_{n,2^{s+1}+2}^{-1}=\mcD_{m+n,A\cdot W_n}(f)\geq \mcD_{m+n,A\cdot W_n}(f\vee \frac{1}{2})+\mcD_{m+n,A\cdot W_n}(f\wedge \frac{1}{2})\geq 2\cdot \frac{1}{4}\tilde R_{n,2^s+2}^{-1},\]
where to see the first inequality it suffices to notice the following estimate concerning the cross energy: $(f(v)-f(v'))^2\geq (f(v)-\frac{1}{2})^2+(f(v')-\frac{1}{2})^2$ when $v\stackrel{m+n}{\sim} v',v\in A_1\cdot W_n,v'\in A_2\cdot W_n$, because $f(v)\leq \frac{1}{2}\leq f(v')$ by (\ref{eqn4add5}).
	
Thus, $R_n(W_{n,2},W_{n,4})\geq \frac{1}{2}\tilde{R}_{n,3} \geq\frac{1}{2^{s+1}}\tilde{R}_{n,2^s+2}$, where the first inequality follows by the case $l = 3$, and the second inequality follows by iterating (\ref{iteration}). Finally, for general $l\geq 3$, the desired estimate follows from the monotonicity of $\tilde{R}_{n,l}$,
\[R_n(W_{n,2}, W_{n,4})\geq \frac{1}{2^{s+2} }\tilde R_{n,2^{s+1}+2}\geq \frac{1}{2^{s+2}}\tilde R_{n,l}\geq \frac {1}{4l}\tilde R_{n,l},\]
providing $2^s+2\leq l<2^{s+1}+2$.
\end{proof}

Another useful observation is that $\sigma_{n+m}$ and $\sigma_n$ are comparable if $m$ is small. In particular, $\sigma_{n+m}\leq 16\cdot\sigma_m\sigma_n$ is another inequality in Theorem 2.1 of Kusuoka and Zhou's paper \cite{KZ}.

\begin{lemma}\label{lemma4add3}
${\frac{N^m}{6k^m}\sigma_n\leq\sigma_{n+m}\leq 16\cdot\sigma_m\sigma_n}\hbox{ for every } m,n\geq 1$.
\end{lemma}
\begin{proof}
For the left side of the inequality, we choose $m'\geq 1, w\stackrel{m'}{\sim} w'$ and $f\in l\big(\{w,w'\}\cdot W_n\big)$ so that $\mcD_{m'+n,\{w,w'\}\cdot W_n}(f)=1$ and $N^n\big([f]_{w\cdot W_n}-[f]_{w'\cdot W_n}\big)^2 = {\sigma_n(w,w') \geq \frac{1}{2}\sigma_n}$ (recall the definition of $\sigma_n$ and Remark 2 after Definition \ref{def31}). We define $f'\in l\big(\{w,w'\}\cdot W_{n+m}\big)$ with the formula $f'(v\cdot \tau)=f(v)\hbox{ for every } v\in \{w,w'\}\cdot W_n,\, \tau \in W_m$. Then one can see $\mcD_{m'+n+m,\{w,w'\}\cdot W_{n+m}}(f')\leq 3k^m$ by a routine argument as before and $N^{n+m}\big([f']_{w\cdot W_{n+m}}-[f']_{w'\cdot W_{n+m}}\big)^2= N^{m}\sigma_n(w,w') \geq  \frac{1}{2}N^m\sigma_n$, which implies  $\sigma_{n+m}\geq \frac{N^m}{2\cdot 3k^m}\sigma_n$.
	
To see the right side of the inequality, we choose $m'\geq 1, w\stackrel{m'}{\sim} w'$ and $g\in l\big(\{w,w'\}\cdot W_{n+m}\big)$ so that $\mcD_{m'+n+m,\{w,w'\}\cdot W_{n+m}}(g)=1$ and $N^{n+m}([g]_{w\cdot W_{n+m}}-[g]_{w'\cdot W_{n+m}})^2 \geq \frac{1}{2}\sigma_{n+m}$. Then we define $g'\in l\big(\{w,w'\}\cdot W_n\big)$ to be $g'(v)=[g]_{v\cdot W_m}\hbox{ for every } v\in \{w,w'\}\cdot W_n$. Then, just as in  Lemma \ref{lemma36}, one can see $\mcD_{m'+n,\{w,w'\}\cdot W_n}(g')\leq 8N^{-m}\sigma_m$ and $N^{n}([g']_{w\cdot W_n}-[g']_{w'\cdot W_n})^2 \geq \frac{1}{2}N^{-m}\sigma_{n+m}$, which implies $\sigma_n\geq \frac{N^{-m}\sigma_{n+m}}{2\cdot 8N^{-m}\sigma_m}$.
\end{proof}

Now, we prove (\ref{eqn45}).

\begin{lemma}\label{lemma48}
	There exists $C>0$ such that for any $n\geq 1$, \[R_{n}(W_{n,2},W_{n,4})=R_{n}(W_{n,1},W_{n,3})\geq C\cdot N^{-n}\sigma_n.\]
\end{lemma}
\begin{proof}
We will prove the lemma for large $n$ only. Then by adjusting the constant $C$, we can make the equation hold for each $n\geq 1$, noticing that $R_n(W_{n,2},W_{n,4})$ (since the minimal energy is always achieved by some non-constant function) and $\sigma_n$ for $n\geq 1$ are all positive and finite. We will fix two positive integers $M_1,M_2$, and prove the desired estimate for $n-M_1-M_2$ for each large $n$.\vspace{0.2cm}

\noindent\emph{The integer $M_1$.} By equations \eqref{eqn31}, \eqref{eqn32} and \eqref{eqn33},
there is $M_1$ such that
\begin{equation}\label{eqn413}
C_1\cdot N^{-n+M_1}\lambda_{n-M_1}<\frac{1}{2}\cdot 12^{-2}\cdot N^{-n}\sigma_n\quad\hbox{ for every } n\geq M_1+1,
\end{equation}
where $C_1$ is the constant in Lemma \ref{lemma35}. \emph{We fix such a $M_1$ in the rest of the proof.}

In fact, for each $n>m\geq 1$, we see that $N^{-n+m}\lambda_{n-m}\leq C_2 R_{m}^{-1}\cdot N^{-n}\lambda_n \leq C_3(k^2N^{-1})^{-m}\cdot N^{-n}\sigma_n$ for some $C_2, C_3 > 0$ depending only on the $\USC$ $K$, where the first inequality follows by equation \eqref{eqn31} and the second inequality follows by equations \eqref{eqn32} and \eqref{eqn33}. It suffices to pick $M_1$ large enough so that $C_1C_3(k^{-2}N)^{M_1}<2\cdot 24^{-2}$.\vspace{0.2cm}

\noindent\emph{The integer $M_2$.} By the same reason as in the last paragraph and noticing that $M_1$ is a fixed integer depending only on $K$ now, we can find $M_2$ such that
\begin{equation}\label{eqn414}
C_1\cdot N^{-n+M_1+M_2}\lambda_{n-M_1-M_2}<\frac{1}{2}({96}N^{M_1})^{-2}\cdot N^{-n}\sigma_n {\quad \hbox{ for every } n \geq M_1 + M_2 + 1,}
\end{equation}
where $C_1$ is the constant in Lemma \ref{lemma35}. \emph{We fix such a $M_2$ in the rest of the proof.}\vspace{0.2cm}

Now we return to the proof. We fix a large $n$ so that
\begin{equation}\label{eqn415}
N^{M_1}\leq  {\frac1{24}}\sqrt{N^{-n}\sigma_n},\ \text{ and }\  n\geq M_1+M_2+1.
\end{equation}
In fact, $\sqrt{N^{-n}\sigma_n}\geq {C_4}(N^{-1}k^{2})^{n/2}$ for some ${C_4 > 0}$ by \eqref{eqn32} and \eqref{eqn33}, so there is $n_0>0$ such that \eqref{eqn415} holds for each $n\geq n_0$. This justifies the meaning of large $n$ that we mentioned at the beginning.
	
Next, by using Remark 2 after Definition \ref{def31}, we choose $w,w'$ in $W_*$ so that $\sigma_n(w,w')\geq \frac{1}{4}\sigma_n$, and choose $f'\in l\big(\{w,w'\}\cdot W_n\big)$ such that $\mcD_{n+|w|,\{w,w'\}\cdot W_n}(f')=1$ and $[f']_{w\cdot W_n}-[f']_{w'\cdot W_n}={\sqrt{N^{-n}\sigma_n(w,w')} \geq \frac{1}{2}\sqrt{N^{-n}\sigma_n}}$. Without loss of generality, we assume $f'$ is antisymmetric with respect to the central point of $\Psi_wK\cup \Psi_{w'}K$ (noticing that $\tilde f':=\frac{f'-\Gamma_{w,w'} f'}{2}$ will decrease the energy, but keep $[\tilde f']_{w\cdot W_n}-[\tilde f']_{w'\cdot W_n}=[f']_{w\cdot W_n}-[f']_{w'\cdot W_n}$, where $\Gamma_{w,w'}$ is the associated rotation). Thus, by letting $f=f'\circ \Psi_w$, we get a function $f\in l(W_n)$ satisfying
\begin{equation}\label{eqn416}
[f]_{W_n}{\geq \frac{1}{4}\sqrt{N^{-n}\sigma_n}},\quad\mcD_n(f)\leq \frac{1}{2}.
\end{equation}
In addition, there is $\tau'\in \partial W_n$ such that $f(\tau')\leq \frac{1}{2}$ (it suffices to choose $\tau'$ so that $\Psi_{w\cdot\tau'}K$  contains the central point of $\Psi_wK\cup \Psi_{w'}K$). We also choose $\tau\in W_n$ such that
\[f(\tau)=\max_{v\in W_n}f(v)\geq[f]_{W_n},\]
where we use the fact that, by \eqref{eqn415} and \eqref{eqn416}, $[f]_{W_n}\geq 6N^{M_1}>1$ for large $n$.

\begin{figure}[htp]
\includegraphics[width=11cm]{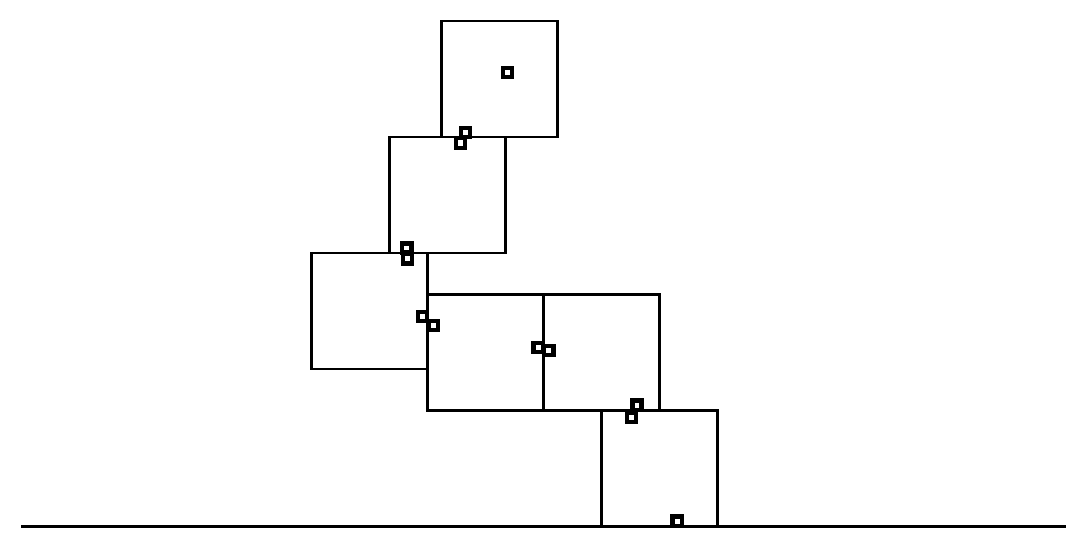}
\caption{The sequence $\tau^{(0)}\cdots \tau^{(2L_n-1)}$ represented by small squares, where we use big squares to represent the associated $M_1$-cells.}
\label{figure3}
\end{figure}
	
Next, we choose a sequence of words $\tau=\tau^{(0)},\cdots,\tau^{(2L_n-1)}=\tau'\in W_n$ such that
\[\begin{cases}
\tau^{(2\ell)}_1\cdots \tau^{(2\ell)}_{M_1}=\tau^{(2\ell+1)}_1\cdots \tau^{(2\ell+1)}_{M_1},&\text{ for }0\leq \ell<L_n,\\
\tau^{(2\ell+1)}\sn \tau^{(2\ell+2)},&\text{ for }0\leq \ell< L_n-1,
\end{cases}\]
and
\[\tau^{(2\ell)}_1\cdots \tau^{(2\ell)}_{M_1}\neq \tau^{(2\ell')}_1\cdots \tau^{(2\ell')}_{M_1}\quad \hbox{ for every } 0\leq \ell\neq \ell'<L_n,\]
where we write each $\tau^{(i)}$ in the form $\tau^{(i)}_1\cdots \tau^{(i)}_n$ with $\tau^{(i)}_j\in \{1,\cdots, N\}$, $1 \leq j \leq n$. Clearly, $L_n\leq N^{M_1}=\#W_{M_1}$. See Figure \ref{figure3} for an illustration for this sequence.

By \eqref{eqn413}, \eqref{eqn416} and Lemma \ref{lemma35}, we see that $|f(\tau^{(0)})-f(\tau^{(1)})|
\leq{2\sqrt{C_1N^{-n+M_1}\lambda_{n - M_1}\mcD_{n}(f)}}\leq \sqrt{\frac{1}{12^2}N^{-n}\sigma_n} \leq \frac{1}{3}[f]_{W_n}$, so $f(\tau^{(1)})>\frac{2}{3}[f]_{W_n}$ by the choice of $\tau=\tau^{(0)}$. In addition, we notice that $f(\tau^{(2L_n-1)})\leq \frac{1}{2}$ and $\big|f(\tau^{(2\ell+1)})-f(\tau^{(2\ell+2)})\big|\leq 1\hbox{ for every } 0\leq\ell<L_n-1$ by the energy estimate in \eqref{eqn416}, so
\[\begin{aligned}
&\frac{2}{3}[f]_{W_n}<f(\tau^{(1)})\leq \sum_{i=1}^{2L_n-2}|f(\tau^{(i)})-f(\tau^{(i+1)})|+|f(\tau^{(2L_n-1)})|\\
\leq&\sum_{\ell=1}^{L_n-1}|f(\tau^{(2\ell)})-f(\tau^{(2\ell+1)})|+(L_n-1)+\frac{1}{2}\leq \sum_{\ell=1}^{L_n-1}|f(\tau^{(2\ell)})-f(\tau^{(2\ell+1)})|+\frac{1}{6}[f]_{W_n},
\end{aligned}\]
where in the last inequality we use \eqref{eqn415}, \eqref{eqn416} and $L_n\leq N^{M_1}$. As a consequence we can find $1\leq \ell<L_n$ such that $\big|f(\tau^{(2\ell)})-f(\tau^{(2\ell+1)})\big|\geq \frac{1}{2(L_n-1)}[f]_{W_n}\geq \frac{1}{2}N^{-M_1}[f]_{W_n}$. Now we fix this $\ell$ and consider
\[g=f\circ \Psi_{\tau^{(2\ell)}_1\cdots \tau^{(2\ell)}_{M_1}}\in l(W_{n-M_1}).\]
Clearly, $\mcD_{n-M_1}(g) \leq \mathcal{D}_{n}(f)\leq \frac12$, and we can choose $\eta,\eta'$ from
\[\big\{\tau_{M_1+1}^{(2\ell)}\cdots \tau_n^{(2\ell)},\text{ } \tau_{M_1+1}^{(2\ell+1)}\cdots \tau_n^{(2\ell+1)}\big\}\bigcup \big\{w\in W_{n-M_1}:q_i\in \Psi_wK \text{ for some }1\leq i\leq 4\big\},\]
such that $g(\eta)-g(\eta')\geq \frac{1}{6}N^{-M_1}[f]_{W_n}$ and $\{\eta,\eta'\}\subset W_{n-M_1,j}$ for some $1\leq j\leq 4$. Again, by Lemma \ref{lemma35} and the choice of $M_2$ \eqref{eqn414}, we have
\[\big|g(v)-g(\eta)\big|\leq \frac{1}{24}N^{-M_1}[f]_{W_n},\quad \big|g(v')-g(\eta')\big|\leq \frac{1}{24}N^{-M_1}[f]_{W_n},\]
for any $v,v'\in W_{n-M_1}$ such that $v_1\cdots v_{M_2}=\eta_1\cdots\eta_{M_2}$ and $v'_1\cdots v'_{M_2}=\eta'_1\cdots \eta'_{M_2}$. Thus, by letting $\kappa=\eta_1\cdots \eta_{M_2}$ and $\kappa'=\eta'_1\cdots\eta'_{M_2}$, we have
\[\min_{v\in \kappa\cdot {W}_{n-M_1-M_2}}g(v)-\max_{v'\in \kappa'\cdot {W}_{n-M_1-M_2}}g(v')\geq g(\eta)-g(\eta')-2\cdot\frac{1}{24}N^{-M_1}[f]_{W_n}\geq \frac{1}{48}N^{-M_1}\sqrt{N^{-n}\sigma_n}.\]
Observing that $\kappa,\kappa'\in W_{M_2,j}$ and $\mcD_{n-M_1-M_2}(g)\leq \frac12$, by Remark 1 after Definition \ref{def31}, we conclude that $\tilde R_{n-M_1-M_2,l}\geq 2\cdot (\frac{1}{48}N^{-M_1})^2\cdot N^{-n}\sigma_n$, where $l$ is the length of the finite chain in $W_{M_2,j}$ connecting $\kappa, \kappa'$. Since $l$ is bounded above by $k^{M_2}$,  by using Lemma \ref{lemma47}, one can see that $R_{n-M_1-M_2}(W_{n-M_1-M_2,2},W_{n-M_1-M_2,4})\geq C_5\cdot N^{-n}\sigma_n$ for some $C_5>0$ depending on $M_1,M_2$.
	
Finally, we finish the proof by applying Lemma \ref{lemma4add3} to see that
\[R_{n-M_1-M_2}(W_{n-M_1-M_2,2}, W_{n-M_1-M_2,4})\geq C\cdot N^{-n+M_1+M_2}\sigma_{n-M_1-M_2}\]
for some $C>0$ depending only on $M_1,M_2$ (and $K$).
\end{proof}
Finally, we prove (\ref{eqn44}) using Lemma \ref{lemma48}.
\begin{corollary}\label{coro49}
There exists $C>0$ such that $C\cdot N^{-n}\sigma_{n}\geq (\frac{2}{k})^mN^{-n-m}\sigma_{n+m}\hbox{ for every } n,m\geq 1$.
\end{corollary}
\begin{proof}
We can find $2^m$ disjoint chains of $m$-cells connecting $W_{m,2}$ and $W_{m,4}$: for each $\bm{e}=e_1e_2\cdots e_m\in \{1,3\}^m:=\{e_1e_2\cdots e_m:e_l=1\text{ or }3\hbox{ for every } 1\leq l\leq m\}$, we consider
\[
\prod_{l=1}^m W_{1,e_l}:=\{w\in W_m: w_l\in W_{1,e_l}\hbox{ for every } 1\leq l\leq m\}.
\]
We write $\prod_{l=1}^m W_{1,e_l}=\{w_{\bm{e},1},w_{\bm{e},2},\cdots,w_{\bm{e},k^m}\}$, where we order the cells from left to right so that $w_{\bm{e},1}\in W_{m,4}, w_{\bm{e},k^m}\in W_{m,2}$ and $w_{\bm{e},l}\sm w_{\bm{e},l'}$ if and only if $|l-l'| = 1$. Then, for each $h\in l(W_m)$ such that $h|_{W_{m,4}}=0$ and $h|_{W_{m,2}}=1$, we have the lower bound estimate of $\mcD_{m}(h)$,
\[\begin{aligned}
\mcD_m(h)&\geq \sum_{\bm{e}\in \{1,3\}^m}\sum_{l=1}^{k^m-1}\big(h(w_{\bm{e},l+1})- h(w_{\bm{e},l})\big)^2\\
&\geq \sum_{\bm{e}\in \{1,3\}^m}\frac{1}{k^m-1}\Big(\sum_{l=1}^{k^m-1}\big(h(w_{\bm{e},l+1})- h(w_{\bm{e},l})\big)\Big)^2\\
&=\frac{1}{k^m-1}\sum_{\bm{e}\in \{1,3\}^m}\big(h(w_{\bm{e},1})-h(w_{\bm{e},k^m})\big)^2 =\frac{2^m}{k^m-1}>\frac{2^m}{k^{m}},
\end{aligned}\]
which means that $R_m(W_{m,2},W_{m,4})\leq (\frac{k}{2})^m$. The lemma then follows by $\sigma_{n + m} \leq 16\sigma_n\sigma_m \leq 16 C_1\sigma_n N^{m}R_m(W_{m,2},W_{m,4}) \leq 16 C_1\sigma_n N^m(\frac{k}{2})^m$ for some $C_1 > 0$, using Lemma \ref{lemma4add3} and Lemma \ref{lemma48}.
\end{proof}

\begin{proof}[Proof of \textbf{(B)}]
The condition \textbf{(B)} is a consequence of Proposition \ref{prop46}, Lemma \ref{lemma48} and Corollary \ref{coro49}.
\end{proof}

\section{Self-similar forms}\label{sec5}
In this section, we prove the existence of a self-similar Dirichlet form on $\USC$. The proof is based on the existence of a limit form $(\bar{\mathcal{E}},\bar{\mathcal{F}})$ in Theorem \ref{thm34}.

\begin{definition}\label{def51}
Let $K$ be a $\USC$ and $\mu$ be the normalized $d_H$-dimensional Hausdorff measure on $K$. A regular Dirichlet form $(\mcE,\mcF)$ on $K$ is called \emph{self-similar}, if
\[\mathcal{F}\cap C(K) = \{f\in C(K): f\circ \Psi_i\in \mathcal{F}\hbox{ for every } 1 \leq i \leq N\}\]
and the self-similar identity of the energy holds,
\begin{equation}
\mcE(f)=\sum_{i=1}^N r^{-1}\mcE(f\circ \Psi_i)\quad\hbox{ for every } f\in \mcF,
\end{equation}
where $r>0$ is called the \emph{renormalization factor}.
\end{definition}

In Kusuoka-Zhou's original strategy \cite{KZ}, the existence of $(\mathcal E,\mathcal F)$ follows by two steps: first, they construct a local regular Dirichlet form $(\bar{\mathcal E},\bar{\mathcal F})$ which is a  limit of $ r^{-n}{\mathcal D}_n$ (recall the constant $r$ in Lemma \ref{lemma33}); then, they construct $(\mathcal E,\mathcal F)$ from $(\bar{\mathcal{E}}, \bar{\mathcal{F}})$ by the way of taking a subsequential limit of Ces\`{a}ro means. The first step has been accomplished in Theorem \ref{thm34}.

\begin{theorem}\label{thm52}
Let $K$ be a $\USC$ and $\mu$ be the $d_H$-dimensional normalized Hausdorff measure on $K$. There is a strongly local, regular, irreducible, $D_4$-symmetric, self-similar Dirichlet from $(\mcE,\mcF)$ on $L^2(K,\mu)$.

In addition, $\mathcal{F} = \bar{\mathcal{F}}$ and there are constants $C,C'>0$ such that \[C\cdot\bar{\mcE}(f)\leq\mcE(f)\leq C'\cdot\bar{\mcE}(f)\quad \hbox{ for every } f\in\mcF,\] where $(\bar \mcE,\bar{\mcF})$ is the same as that in Theorem \ref{thm34}.
\end{theorem}

\noindent\textbf{Remark.}  Recall \cite[Theorem 6.9]{KZ}, $(\mcE,\mcF)$ is defined to be a limit of Ces\`{a}ro means of bilinear forms $(\hat{\mcE}_n,\bar{\mcF})$, $n\geq 1$ which are given by
\[\hat{\mcE}_n(f,g)=r^{-n}\sum_{w\in W_n}\sum_{i=1}^4\big(P_nf(w)-P_{n}^{L_i}f(w)\big) \big(P_ng(w)-P_{n}^{L_i}g(w)\big)\quad \hbox{ for every } f,g\in\bar{\mcF},\]
with $P_n^{L_i}f(w)=\int_{L_i}f\circ \Psi_w d\nu_i$, and $\nu_i$ the Lebesgue measure on $L_i$.
However, the Markov property of the limit form is difficult to prove. Kigami filled up this gap in \cite{ki25}   by developing a fixed point theorem of order preserving additive maps on an ordered topological cone. Using this technique, the existence of a desired Dirichlet form $(\mcE,\mcF)$ follows, especially with the Markov property (see \cite[Theorem 2.3]{ki25}). For the self-containedness of the paper, we will still provide a proof following Kusuoka-Zhou's strategy, but replacing $\hat{\mcE}_n(f,g)$ to $\sum_{w\in W_n}r^{-n}\bar \mcE(f\circ \Psi_w, g\circ \Psi_w)$, which is also essentially similar to that of Kigami \cite{ki25}.

\begin{proof}
Recall the form $(\bar{\mcE},\bar{\mcF})$ in Theorem \ref{thm34}, and notice that $\bar{\mcF}\subset C(K)$. For $n\geq 0$, define
\[\tilde{\mcE}_n(f)=\sum_{w\in W_n}r^{-n}\bar{\mcE}(f\circ \Psi_w)\quad\hbox{ for every } f\in \mcF_n,\]
where $\mcF_n=\{f\in C(K):f\circ \Psi_w\in \bar{\mcF}\hbox{ for every } w\in W_n\}$. It is straightforward to check that $\bar{\mcF}\subset\mcF_1$. In fact, $\sup_{m\geq 0}r^{-m}\mcD_m(f\circ \Psi_i)<\infty$ for each $f\in\bar{\mcF}$ and ${i\in W_1}$, since $\mcD_{m+1}(f)\geq \sum_{i\in W_1}\mcD_m(f\circ \Psi_i)\hbox{ for every } m\geq 0$. It then follows easily that $\bar{\mcF}=\mcF_0\subset\mcF_1\subset\mcF_2\subset\cdots$, since if $f\in \mcF_n,n\geq 0$, then $f\circ \Psi_w\in \bar{\mcF}\subset \mcF_1\hbox{ for every } w\in W_n$, and hence $f\circ \Psi_w\circ\Psi_i\in \bar{\mcF}\hbox{ for every } w\in W_n,\, i\in W_1$, which implies $f\in \mcF_{n+1}$.\vspace{.2cm}

Next we claim that
\begin{equation}\label{eqn52}
C_1\cdot r^{-n}\mcD_n(f)\leq \tilde{\mcE}_n(f)\leq C_2\cdot \liminf_{m\to\infty}r^{-m}\mcD_m(f)\quad\hbox{ for every } f\in\mcF_n,
\end{equation}
for some $C_1,C_2>0$ independent of $n$. On one hand,  by Theorem \ref{thm34}, there is $C_2>0$ such that $\bar{\mcE}(f)\leq C_2\cdot\liminf\limits_{m\to\infty}r^{-m}\mcD_m(f)$ for any $f\in \bar{\mcF}$. It follows from the definition of $\mcD_m$, for any $f\in \mcF_n$, we have
\[\begin{aligned}
\tilde{\mcE}_n(f)=\sum_{w\in W_n}r^{-n}\bar{\mcE}(f\circ \Psi_w)&\leq C_2\cdot\liminf_{m\to\infty}\sum_{w\in W_n}r^{-n-m}\mcD_m(f\circ \Psi_w)\\
&\leq C_2\cdot\liminf_{m\to\infty}r^{-n-m}\mcD_{n+m}(f)\\
&= C_2\cdot\liminf_{m\to\infty}r^{-m}\mcD_m(f).
\end{aligned}\]
On the other hand, still by Theorem \ref{thm34}, $\bar{\mcF}\subset C(K)$ and $\big|P_0f-f(x)\big|^2\leq C_3\cdot\bar{\mcE}(f)\hbox{ for every } x\in K, \, f\in \bar{\mcF}$, for some $C_3>0$, where $P_0f=\int_K f(y)\mu(dy)$. Thus, for any $f\in \mcF_n$ and $w\sn w'$ in $W_n$,
\[\big|(P_nf)(w)-(P_nf)(w')\big|^2\leq 2C_3\cdot\big(\bar{\mcE}(f\circ \Psi_w)+\bar{\mcE}(f\circ \Psi_{w'})\big),\]
where $(P_nf)(w)=N^{n}\int_{\Psi_wK}f(y)\mu(dy)$ as in previous sections. By summing over all $w\sn w'$, we get $C_1\cdot r^{-n}\mcD_n(f)\leq \tilde{\mcE}_n(f)$ for some $C_1>0$.

As a consequence of (\ref{eqn52}), we have
\begin{equation}\label{eqn53}
C_1\cdot \sup\limits_{m\geq n}r^{-m}\mcD_m(f)\leq \sup_{m\geq n}\tilde\mcE_m(f)\leq C_2\cdot \liminf\limits_{m\to\infty}r^{-m}\mcD_m(f)\quad\hbox{ for every } f\in\mcF_{n}, \, n\geq0.
\end{equation}

Now, let's show $\bar{\mcF}=\mcF_1=\mcF_2\cdots$. For each $n\geq 1$, one can see
\[\begin{aligned}
\bar{\mcF}=\{f\in \mcF_n:\sup_{m\geq n}r^{-m}\mcD_m(f)<\infty\}&=\{f\in \mcF_n:\sup_{m\geq n}\tilde\mcE_m(f)<\infty\}\\
&=\{f\in \mcF_n:\sup_{m\geq n}\sum_{w\in W_n}r^{-n}\tilde\mcE_{m-n}(f\circ\Psi_w)<\infty\}\\
&=\{f\in \mcF_n:\sup_{m\geq 0}\tilde\mcE_m(f\circ \Psi_w)<\infty\hbox{ for every } w\in W_n\}\\
&=\{f\in \mcF_n:\sup_{m\geq 0}r^{-m}\mcD_m(f\circ \Psi_w)<\infty\hbox{ for every } w\in W_n\}=\mcF_n,
\end{aligned}\]
where the second and fifth equalities follow from \eqref{eqn53}.\vspace{.2cm}
	
From now on, we take $\mathcal{F} = \bar{\mathcal{F}}$. Let $\hat{\mcF}$ be a $\mathbb{Q}$-vector subspace of $\mcF$ with countable number of elements, that is dense in $\mcF$ with respect to the norm $\|f\|_{\bar{\mcE}_1}:=\sqrt{\bar{\mcE}(f)+\|f\|^2_{L^2(K,\mu)}}$. To achieve this, one can simply choose a $\mathbb{Q}$-vector dense subspace $H$ of $L^2(K,\mu)$ with countable number of elements, and let $\hat{\mcF}=U_1(H)$, where $U_1$ is the resolvent operator associated with $\bar{\mcE}_1$, i.e. $\bar{\mcE}_1(U_1f,g)=\int_{K}fgd\mu$, for any $f\in L^2(K,\mu)$ and $g\in \mcF$.
	
Then by a diagonal argument, there is a subsequence $\{n_l\}_{l\geq 1}$ such that the limit
\begin{equation}\label{eqn54}
\mcE(f):=\lim\limits_{l\to\infty} \frac{1}{n_l}\sum_{m=1}^{n_l}\tilde{\mcE}_m(f)
\end{equation}
exists for any $f\in \hat\mcF$. In addition, by (\ref{eqn52}) and Theorem \ref{thm34}, we know that there exist $C_4,C_5>0$ such that
\begin{equation}\label{eqn55}
C_4\cdot\bar{\mcE}(f)\leq \frac1n\sum_{m=1}^{n}\tilde{\mcE}_m(f)\leq C_5\cdot\bar{\mcE}(f)\quad\hbox{ for every }f\in \mcF,n\geq 1.
\end{equation}
Hence by \eqref{eqn54}, for every $f\in \hat\mcF$,
\begin{equation}\label{eqn56}
C_4\cdot\bar{\mcE}(f)\leq \mcE(f)\leq C_5\cdot\bar{\mcE}(f).
\end{equation}

Now we show that \eqref{eqn54} actually holds for every $f\in\mcF$, i.e. for $f\in \mcF$, we will prove that the limit in the right hand side of \eqref{eqn54} exists. Indeed, for any $\varepsilon>0$ and large enough $l,l'$, by using \eqref{eqn54}, we see that
\begin{align*}
&\Big|\sqrt{\frac{1}{n_l}\sum_{m=1}^{n_l}\tilde{\mcE}_m(f)}-\sqrt{\frac{1}{n_{l'}}\sum_{m=1}^{n_{l'}}\tilde{\mcE}_m(f)}\Big|\\
\leq &\sqrt{\frac{1}{n_l}\sum_{m=1}^{n_l}\tilde{\mcE}_m(f-f')}+\Big|\sqrt{\frac{1}{n_l}\sum_{m=1}^{n_l}\tilde{\mcE}_m(f')}-\sqrt{\frac{1}{n_{l'}}\sum_{m=1}^{n_{l'}}\tilde{\mcE}_m(f')}\Big|+\sqrt{\frac{1}{n_{l'}}\sum_{m=1}^{n_{l'}}\tilde{\mcE}_m(f'-f)}\\
\leq &\varepsilon,
\end{align*}
where in the second line we choose $f'\in \hat{\mcF}$ so that $\sqrt{\bar{\mcE}(f-f')}<\frac{\varepsilon}{3\sqrt{C_5}}$, and we use \eqref{eqn54} and \eqref{eqn55} in the third inequality. This immediately implies that for every $f\in \mcF$, $\mathcal E(f):=\lim_{l\to\infty}\frac{1}{n_l}\sum_{m=1}^{n_l}\tilde{\mcE}_m(f)$ exists. In addition,   \eqref{eqn56} also holds for every $f\in \mcF$ by (\ref{eqn55}).

Clearly, the limit form $\mcE$ keeps the properties: $\mcE(tf)=t^2\mcE(f)$, $\mcE(f+g)+\mcE(f-g)=2\mcE(f)+2\mcE(g)$ and $\mcE\big((f\vee0)\wedge1\big)\leq \mcE(f)$ for any $t\in\mathbb{R}$ and $f,g\in \mcF$. Hence, the functional $\mcE$ induces a densely defined closed symmetric bilinear form $(\mcE,\mcF)$ on $K$.

Moreover, since $(\frac1n\sum_{m=1}^{n}\tilde{\mcE}_m,\mcF)$ satisfies the Markov property and $D_4$-symmetry by Theorem \ref{thm34} (b) for every $n\geq 1$, the limit form $(\mcE,\mcF)$ also satisfies the Markov property and $D_4$-symmetry. Hence, $(\mcE,\mcF)$ is a $D_4$-symmetric regular Dirichlet form.
		
Next, we show the self-similarity of the form. The property $\mcF=\mcF_1=\{f\in C(K):\,f\circ \Psi_i\in\mcF\hbox{ for every }1\leq i\leq N\}$ has already been verified, and by using \eqref{eqn54} and \eqref{eqn56}, we  see that
\begin{align*}
\mcE(f)=\lim\limits_{l\to\infty}\frac{1}{n_l}\sum_{m=1}^{n_l}\tilde{\mcE}_{m}(f)=&\sum_{i=1}^N\lim\limits_{l\to\infty}\frac{1}{n_l}\sum_{m=0}^{n_l-1}r^{-1}\tilde{\mcE}_{m}(f\circ\Psi_i)\\
=&\sum_{i=1}^Nr^{-1}\lim\limits_{l\to\infty}\frac{1}{n_l}\sum_{m=1}^{n_l}\tilde{\mcE}_{m}(f\circ\Psi_i)=\sum_{i=1}^Nr^{-1}\mcE(f\circ\Psi_i)
\end{align*}
holds for every $f\in\mcF$, where the third equality is due to $\lim_{l\to \infty}\frac{1}{n_l}\tilde{\mcE}_{n_l}(f\circ \Psi_i)=0$ as a consequence of (\ref{eqn52}).
	
The strongly local property follows from the self-similarity of the form. Finally, since $\mcF\subset C(K)$, $1_A\in \mcF$ implies $\mu(A)\in \{0,1\}$. So $(\mcE,\mcF)$ is irreducible by \cite[Theorem 1.6.1]{FOT}, noticing that $1\in \mcF$.
\end{proof}

\noindent\textbf{Remark.} Since $1\in \mcF$ and $\mcE(1)=0$, by \cite[Theorem 1.6.3]{FOT}, $(\mcE,\mcF)$ is also recurrent.

\appendix
\renewcommand{\appendixname}{Appendix~\Alph{section}}

\section{Proof of Proposition \ref{prop32}}\label{AppendixA}
In this appendix, we prove Proposition \ref{prop32}. The proof is essentially the same as that in \cite[Theorem 2.1]{KZ}, with slight adjustment due to the change of condition \textbf{(A3)}. We reproduce it here for the convenience of readers. In the following context, all the positive constants $C$ appeared depend only on the $\USC$.

We begin with an observation that all the constants $\lambda_n,\sigma_n,R_n$ are positive and finite, and in addition, $R_n$ are bounded from below by a multiple of $(k^{2}N^{-1})^n$.

\begin{lemma}[\cite{KZ}, Proposition 2.7]\label{lemmaA1}
The constants $\lambda_n,\sigma_n$ and $R_n$ for $n\geq 1$ are positive and finite. In addition, there exists $C>0$ so that  (\ref{eqn32}) holds:
\[R_n\geq C\cdot (k^2N^{-1})^n\quad \hbox{ for every } n\geq 1. \]
\end{lemma}
\begin{proof}
To see $\lambda_n>0$, we let $f\in l(W_n)$ be a function that is not a constant, then $\sum_{w\in W_n}\big(f(w)-[f]_{W_n}\big)^2>0$ (as we can find $w\in W_n$ such that $f(w)\neq [f]_{W_n}$ if $f$ is not a constant) and $\mcD_{n}(f)<\infty$ (as $\mcD_n(f)$ is a finite summation of terms of the form $\big(f(w)-f(w')\big)^2$ and each term is finite). So $\lambda_n\geq \frac{\sum_{w\in W_n}(f(w)-[f]_{W_n})^2}{\mcD_{n}(f)}>0$.

To see $\sigma_n>0$, we choose $m\geq 1$ and $w\sm w'$, and let $f\in l(\{w,w'\}\cdot W_n)$ such that $f|_{w\cdot W_n}=1$ and $f|_{w'\cdot W_n}=0$. Then $[f]_{w\cdot W_n}-[f]_{w'\cdot W_n}=1$ and $\mcD_{m+n,\{w,w'\}\cdot W_n}(f)\leq 3k^n$. So $\sigma_n\geq \frac{N^n([f]_{w\cdot W_n}-[f]_{w'\cdot W_n})^2}{\mcD_{m+n,\{w,w'\}\cdot W_n}(f)}>0$.

To see that $R_n$ is finite, we choose $m\geq 1$ and $w\in W_m$. Then, by Remark 1 after Definition \ref{def31}, we see that $R_n\leq R_{m+n}(w\cdot W_n,\mathcal{N}^c_w\cdot W_n)<\infty$.
	

Next, for each $n\geq 1$, we show that $\sigma_n$ is finite. We fix $m\geq 1$, $w\sm w'$ in $W_m$ and $f\in l(W_{m+n})$ with $\mcD_{m+n, \{w,w'\}\cdot W_n}(f)=1$. We notice that $\{w,w'\}\cdot W_n$ is connected in the sense that for every distinct pair $v, v'$ in $\{w,w'\}\cdot W_n$ we can find a finite path of the form $v=v^{(1)},v^{(2)},\cdots,v^{(l)}=v'\in\{w,w'\}\cdot W_n$ such that $v^{(i)}\stackrel{m+n}{\sim} v^{(i+1)}$ for  $1\leq i<l$. In addition, we can assume that the path is \emph{self-avoiding} (by erasing loops from the original path if necessary), so $l\leq 2N^n$ as there are only $2N^n$ points in $\{w,w'\}\cdot W_n$. Thus, for each $u,u'\in W_n$, we have
\[\begin{aligned}
\big(f(w\cdot u)-f(w'\cdot u')\big)^2&=\big(\sum_{i=1}^{l-1}f(v^{(i)})-f(v^{(i+1)})\big)^2\\
&\leq (l-1)\sum_{i=1}^{l-1}\big(f(v^{(i)})-f(v^{(i+1)})\big)^2\leq 2N^n\mcD_{m+n, \{w,w'\}\cdot W_n}(f)=2N^n,
\end{aligned}\]
where we fix a self-avoiding path $w\cdot u=v^{(1)},v^{(2)},\cdots,v^{(l)}=w'\cdot u'\in\{w,w'\}\cdot W_n$ such that $v^{(i)}\stackrel{m+n}{\sim} v^{(i+1)}$ for each $i=1,2,\cdots l-1$ in the first equality, and we use the facts that $l\leq 2N^n$ and each edge is used only once (since the path is self-avoiding) in the last inequality. By taking the summation over all $u,u'\in W_n$, we get
\[\begin{aligned}
\big([f]_{w\cdot W_n}-[f]_{w'\cdot W_n}\big)^2&{=\big(\frac{1}{N^{2n}}\sum_{u,u'\in W_n}\big(f(w\cdot u)-f(w'\cdot u')\big)\big)^2}\\&\leq\frac{1}{N^{2n}}\sum_{u,u'\in W_n}\big(f(w\cdot u)-f(w'\cdot u')\big)^2\\
&\leq\frac{1}{N^{2n}}\cdot N^{2n}\cdot 2N^{n}=2N^{n}.
\end{aligned}\]
Thus $\sigma_n(w,w')\leq 2N^{2n}$ since the estimate works for each $f\in l(W_{m+n})$ with $\mcD_{m+n, \{w,w'\}\cdot W_n}(f)=1$. The same argument works for each $m\geq 1$ and $w\sm w'$, so we finally see $\sigma_n\leq 2N^{2n}$.

To see $\lambda_n$ is always finite, we take any function $f\in l(W_n)$ with $\mcD_n(f)=1$. Then for any $w,w'\in W_n$,
\[\big(f(w)-f(w')\big)^2\leq N^n\mcD_n(f)=N^n,\]
as we can find a self-avoiding path of the form $w=w^{(1)},w^{(2)},\cdots,w^{(l)}=w'\in W_n$ with $l\leq N^n$ such that $w^{(i)}\sim w^{(i+1)}\hbox{ for every }1\leq i< l$. Hence
\[\begin{aligned}
\sum_{w\in W_n}\big(f(w)-[f]_{W_n}\big)^2&=\sum_{w\in W_n}\Big(\frac{1}{N^n}\sum_{w'\in W_n}\big(f(w)-f(w')\big)\Big)^2\\
&\leq \frac{1}{N^n}\sum_{w,w'\in W_n}\big(f(w)-f(w')\big)^2\leq \frac{1}{N^n}\cdot N^{2n}\cdot N^n=N^{2n},
\end{aligned}\]
which gives that $\lambda_n$ is finite for all $n\geq 1$.
	
Now, we prove (\ref{eqn32}) by closely following the construction of \cite{KZ}. We remark that (\ref{eqn32}) can also be derived as an immediate consequence of \cite[Lemma 4.6.15]{ki5}. For $m\geq 1$, $w\in W_m$, we let
\[f'(x)=\max\big\{0,1-\frac{d(x,\Psi_wK)}{c_0k^{-m}}\big\}\in C(K),\]
where $c_0$ is the constant in \textbf{(A3)}. Then, $f'|_{\Psi_wK}=1$ and $f'|_{\bigcup_{w'\in  \mathcal{N}_w^c}\Psi_{w'}K}=0$, noticing that $d(\Psi_wK,\Psi_{w'}K)\geq c_0k^{-m}\hbox{ for every } w'\in\mathcal{N}_w^c$. We define $f=P_{m+n}f'$, i.e.
\[f(v)=N^{m+n}\int_{\Psi_vK}f'(x)\mu(dx)\quad\hbox{ for every } v\in W_{m+n}.\]
Then, if $v\stackrel{m+n}{\sim} v'$, we have
\[\big|f(v)-f(v')\big|\leq c_0^{-1}k^m\cdot\sup_{x\in \Psi_vK, y\in\Psi_{v'}K}d(x,y)\leq 2\sqrt{2}c_0^{-1}k^{-n}.\]
Thus
\[\begin{aligned}
\mcD_{m+n}(f)&=\sum_{v\stackrel{m+n}{\sim} v'}\big(f(v)-f(v')\big)^2 \\
&\leq\sum_{v\in \mathcal{N}_w\cdot W_n}\sum_{v': v'\stackrel{m+n}{\sim} v}\big(f(v)-f(v')\big)^2\leq 64N^n\cdot 8\cdot 8c_0^{-2}k^{-2n},
\end{aligned}\]
where in the first inequality we use the fact that $|f(v)-f(v')|>0$ only if $f(v)>0$ or $f(v')>0$, in the second inequality $64$ comes as an upper bound of $\#\mathcal{N}_w$ and the middle $8$ is an upper bound of the number of neighbours of $v$. Thus noticing that $f|_{w\cdot W_n}=1$ and $f|_{\mathcal{N}^c_w\cdot W_n}=0$, we have
\[R_n\big(w\cdot W_n, \mathcal{N}^c_w\cdot W_n\big)\geq 2^{-12}c_0^2(k^2N^{-1})^n.\]
This gives the desired estimate.
\end{proof}

For $l\geq 1$ and $w\in W_n, n\geq 1$, let us define
\[\begin{aligned}
\mathcal{N}_{l,w} =\big\{w'\in W_{n}:\text{there exist }w(i), 0\leq i\leq l,&w(0)=w,w(l)=w',\\
&\Psi_{w(i)}K\cap \Psi_{w(i+1)}K\neq \emptyset\hbox{ for every } 0\leq i< l\big\}.
\end{aligned}\]
In particular, note that $\mathcal{N}_{2,w}=\mathcal{N}_w$ in Definition \ref{def31} (b).

Next, we prove the inequality $\lambda_nN^m R_m\leq C\cdot \lambda_{n+m}$ in Proposition \ref{prop32}, which might be affected by the change of \textbf{(A3)}.

\begin{lemma}\label{lemmaa1}
Let $m,n\geq 1$, and $\{g_w\}_{w\in W_n}$ be a collection of non-negative functions in $l(W_{n+m})$ such that $\sum_{w\in W_n}g_w=1$ and $g_w(v)=0$ for every $w\in W_n, v\in \mathcal{N}_w^c\cdot W_m$.
	
Then for each $f\in l(W_{n})$, we have
\[\mcD_{n+m}(\tilde{f})\leq C\cdot\big(\max_{w\in W_n}\mcD_{n+m}(g_w)\big)\mcD_n(f),\]
for some constant $C>0$, where $\tilde{f}=\sum_{w\in W_n}f(w)g_w\in l(W_{n+m})$.
\end{lemma}
\begin{proof}
For any $v\stackrel{n+m}{\sim}v'$ in  $W_{n+m}$, we write $S(v,v')=\big\{w\in W_n:g_w(v)+g_w(v')>0\big\}$ and $a(v,v')=[f]_{S(v,v')}$. Note that $\#S(v,v')$ is uniformly bounded depending only on the $\USC$. Then,
\[\begin{aligned} \mcD_{n+m}(\tilde{f})&=\sum_{v\stackrel{n+m}{\sim}v'}\big(\tilde{f}(v)-\tilde{f}(v')\big)^2\\
&=\sum_{v\stackrel{n+m}{\sim}v'}\Big(\sum_{w\in S(v,v')}\big(f(w)-a(v,v')\big)\big(g_w(v)-g_w(v')\big)\Big)^2\\
&\leq C_1\cdot \sum_{v\stackrel{n+m}{\sim}v'}\sum_{w\in S(v,v')}\big(f(w)-a(v,v')\big)^2\big(g_w(v)-g_w(v')\big)^2\\
&=C_1\cdot \sum_{w\in W_n}\sum_{v\stackrel{n+m}{\sim}v':w\in S(v,v')}\big(f(w)-a(v,v')\big)^2\big(g_w(v)-g_w(v')\big)^2,
\end{aligned}\]
where the second equality is due to $\sum_{w\in W_n}g_w=1$, in the third and fourth lines $C_1=\max_{v\stackrel{n+m}{\sim}v'}\#S(v,v')$, and in the fourth line $\sum_{v\stackrel{n+m}{\sim}v':w\in S(v,v')}$ means summation over all (unordered) edges $v\stackrel{n+m}{\sim}v'$ such that $w\in S(v,v')$. To continue the estimate, we notice that
\[\big((f(w)-a(v,v')\big)^2\leq \max_{w'\in \mathcal{N}_{5,w}}\big(f(w)-f(w')\big)^2\leq 5\cdot\mcD_{n,\mathcal{N}_{5,w}}(f),\]
where $5$ appears since each $w'$ in $S(v,v')$ locates in $\mathcal N_{5,w}$.
Thus, we have
\[\begin{aligned}
\mcD_{n+m}(\tilde{f})&\leq 5C_1\cdot \sum_{w\in W_n}\mcD_{n,\mathcal{N}_{5,w}}(f)\sum_{v\stackrel{n+m}{\sim}v':w\in S(v,v')}\big(g_w(v)-g_w(v')\big)^2\\
&\leq 5C_1\cdot\sum_{w\in W_n}\mcD_{n,\mathcal{N}_{5,w}}(f)\big(\max_{w\in W_n}\mcD_{n+m}(g_w)\big).
\end{aligned}\]
The lemma follows immediately, noticing that $\sum_{w\in W_n}\mcD_{n,N_{5,w}}(f)\leq C_2\cdot \mcD_n(f)$ for some $C_2>0$.
\end{proof}

\begin{lemma}\label{lemmaa2}
Let $m,n\geq 1$. There exists a collection of non-negative functions $\{g_w\}_{w\in W_{n}}$ in  $l(W_{n+m})$ such that $\sum_{w\in W_{n}}g_w=1$, $g_w(v)=0\hbox{ for every } w\in W_{n},v\in \mathcal{N}_w^c\cdot W_{m}$, and
\[\mcD_{n+m}(g_w)\leq C\cdot R_m^{-1}\quad\hbox{ for every } w\in W_{n},\]
\noindent for some constant $C>0$.
\end{lemma}
\begin{proof}
For each $w\in W_{n}$, let $\tilde{g}_w$ be the unique function supported on $\mathcal{N}_w\cdot W_{m}$ such that
\[\mcD_{n+m}(\tilde{g}_w)=R_{n+m}^{-1}\big(w\cdot W_{m},\mathcal{N}_w^c\cdot W_m\big)\leq R_m^{-1},\quad \tilde{g}_w|_{w\cdot W_m}=1.\]
Let $\tilde{g}=\sum_{w\in W_{n}}\tilde{g}_w$ and $g_w=\frac{\tilde{g}_w}{\tilde{g}}$, noticing that $\tilde{g}\geq 1$ (by the Markov property and by Remark 1 after Definition \ref{def31}, we see that $0\leq \tilde g_w\leq 1$). We immediately have $\sum_{w\in W_{n}}g_w=1$ and $g_w(v)=0\hbox{ for every } w\in W_{n},v\in \mathcal{N}_w^c\cdot W_{m}$. It remains to estimate the energy of each $g_w$.
	
First, we notice that $\tilde{g}\geq1$, so $\big|\frac{1}{\tilde{g}(v)}-\frac{1}{\tilde{g}(v')}\big|\leq \big|\tilde{g}(v)-\tilde{g}(v')\big|$. Thus, we have
\[\begin{aligned}
\mcD_{n+m,\mathcal{N}_{3,w}\cdot W_m}(1/\tilde{g})
\leq \mcD_{n+m,\mathcal{N}_{3,w}\cdot W_m}(\tilde{g})
&=\mcD_{n+m,\mathcal{N}_{3,w}\cdot W_m}\big(\sum_{w':\mathcal{N}_{2,w'}\cap \mathcal{N}_{3,w}\neq \emptyset}\tilde{g}_{w'}\big)\\
&\leq C_1\cdot \max_{w':\mathcal N_{2,w'}\cap \mathcal{N}_{3,w}\neq \emptyset}\mcD_{n+m}(\tilde{g}_{w'})
\leq C_1\cdot R_m^{-1},
\end{aligned}\]
for some $C_1>0$,  where the equality is because $\tilde g(v)=\sum_{w':\mathcal{N}_{2,w'}\cap \mathcal{N}_{3,w}\neq \emptyset}\tilde{g}_{w'}(v)\hbox{ for every } v\in \mathcal{N}_{3,w}\cdot W_m$.
As $g_w$ is supported on $\mathcal N_w\cdot W_m=\mathcal{N}_{2,w}\cdot W_m$, $|g_w(v)-g_w(v')|>0$ only if $\{v,v'\}\subset N_{3,w}\cdot W_m$. Thus we  have
\[\begin{aligned}
\sqrt{\mcD_{n+m}(g_w)}&=\sqrt{\mcD_{n+m,\mathcal{N}_{3,w}\cdot W_m}(\tilde{g}_w/\tilde{g})}\\
&\leq \sqrt{\mcD_{n+m,\mathcal{N}_{3,w}\cdot W_m}(\tilde{g}_w)}+\sqrt{\mcD_{n+m,\mathcal{N}_{3,w}\cdot W_m}(1/\tilde{g})}\leq (C^{1/2}_1+1)\cdot R_m^{-1/2},
\end{aligned}\]
where the first inequality is due to $\|\tilde{g}_w\|_{l^\infty(W_{n+m})}\leq 1$ and $\|1/\tilde{g}\|_{l^\infty(W_{n+m})}\leq 1$.
\end{proof}

\begin{proposition}\label{propa3}
	There is a constant $C>0$ such that
	\[\lambda_nN^mR_m\leq C\cdot \lambda_{n+m+2}\qquad \hbox{for every } n,m\geq 1.\]
\end{proposition}
\begin{proof}
	Let $f\in l(W_n)$ so that $\mcD_n(f)=1$ and $\sum_{w\in W_n}\big(f(w)-[f]_{W_n}\big)^2=\lambda_n$.
	Let $\tilde{f}\in l(W_{n+2})$ be defined as $\tilde{f}(w\cdot \tau)=f(w)\hbox{ for every } w\in W_n,\tau\in W_2$. Clearly,
	\[\mcD_{n+2}(\tilde{f})\leq 3k^2\mcD_n(f)=3k^2,\]
	where $3$ comes from the fact that for any pair $w\sn w'$ in $W_n$, each $(n+2)$-subcell in $\Psi_wK$ is neighbouring to at most $3$ $(n+2)$-subcells in $\Psi_{w'}K$.  Next, we let $\{g_w\}_{w\in W_{n+2}}$ in $l(W_{n+m+2})$ be defined as in Lemma \ref{lemmaa2}, and define $\tilde{\tilde{f}}=\sum_{w\in W_{n+2}}\tilde{f}(w)g_w$. By Lemma \ref{lemmaa1} and \ref{lemmaa2}, we know there is a $C_1>0$ such that
	\[\mcD_{n+m+2}(\tilde{\tilde{f}})\leq C_1\cdot R_m^{-1}\mcD_{n+2}(\tilde{f})\leq 3k^2C_1\cdot R_m^{-1}.\]
	
	On the other hand, let $I=W_2\setminus \big(\partial W_2\cup\{w\in W_2:w\stackrel{2}{\sim}\partial W_2\}\big)$. We have $I\neq \emptyset$, and
	\[\tilde{\tilde{f}}(w\cdot \tau)=f(w)\qquad\hbox{for every } w\in W_n,\tau\in I\cdot W_m\]
	 by the definition of $\{g_w\}_{w\in W_{n+2}}$.
	As a consequence, we have,
	\[\sum_{v\in W_{n+m+2}}\big(\tilde{\tilde{f}}(v)-[\tilde{\tilde{f}}]_{W_{n+m+2}}\big)^2\geq (\#I)\cdot N^m\sum_{w\in W_n}\big(f(w)-[f]_{W_n}\big)^2=(\#I)\cdot\lambda_nN^m.\]
	Thus, $\lambda_{n+m+2}\geq \Big(\sum_{v\in W_{n+m+2}}\big(\tilde{\tilde{f}}(v)-[\tilde{\tilde{f}}]_{W_{n+m+2}}\big)^2\Big)/\mcD_{n+m+2}(\tilde{\tilde{f}})\geq C_2\cdot \lambda_nN^mR_m$ for some $C_2>0$.
\end{proof}

\begin{proof}[Proof of Proposition \ref{prop32}.]
By Lemma \ref{lemmaA1}, all the constants $\lambda_n, R_n$ and $\sigma_n$, $n\geq 1$ are positive and finite, and the estimate (\ref{eqn32}) holds. The left hand side of (\ref{eqn31}), $\lambda_nN^mR_m\leq C\cdot\lambda_{n+m}\hbox{ for every } n,m\geq1$, will follow immediately from Proposition \ref{propa3} and the right hand side of (\ref{eqn31}), $\lambda_{n+m}\leq C\cdot\lambda_n\sigma_m\hbox{ for every } n,m\geq 1$. So it remains to show the right hand side of (\ref{eqn31}).\vspace{0.2cm}
	
Let $f\in l(W_{n+m})$ so that $\mcD_{n+m}(f)=1$ and $\sum_{v\in W_{n+m}}\big(f(v)-[f]_{W_{n+m}}\big)^2=\lambda_{n+m}$. Let $f'\in l(W_n)$ be defined as $f'(w)=[f]_{w\cdot W_m}\hbox{ for every } w\in W_n$. Then, by  Lemma \ref{lemma36}, we know that
\[\mcD_n(f')\leq C_1\cdot N^{-m}\sigma_m\mcD_{n+m}(f)=C_1\cdot N^{-m}\sigma_m,\]
for some constant $C_1>0$. On the other hand, we have
\[\begin{aligned}
\lambda_{n+m}&=\sum_{v\in W_{n+m}}\big(f(v)-[f]_{W_{n+m}}\big)^2\\
&=\sum_{w\in W_n}\sum_{v\in w\cdot W_m}\big(f(v)-f'(w)\big)^2+N^m\sum_{w\in W_n}\big(f'(w)-[f']_{W_n}\big)^2\\
&\leq \lambda_m\mcD_{n+m}(f)+N^m\lambda_n\mcD_n(f')\leq \lambda_m+C_1\cdot\lambda_n\sigma_m.
\end{aligned}\]
In addition, by Proposition \ref{propa3} and (\ref{eqn32}), we know that $\lambda_{n+m}\geq 2\cdot\lambda_m\hbox{ for every } n\geq n'$ for some $n'\geq 1$. Thus $\lambda_{n+m}\leq 2C_1\lambda_n\sigma_m$ holds for $m\geq 1,n\geq n'$. For $m\geq 1,1\leq n<n'$, we have
\[\begin{aligned}
\lambda_{n+m}&\leq C_2\lambda_{n'+2+m}(N^{n'-n}R_{n'-n})^{-1}\\
&\leq 2C_1C_2\lambda_{n'+2}\sigma_{m}(N^{n'-n}R_{n'-n})^{-1}\leq C_3\lambda_n\sigma_m,
\end{aligned}\]
for some constant $C_2>0$, where the first inequality follows from Proposition \ref{propa3}, and $C_3=2C_1C_2\cdot\max\limits_{1\leq n<n'}\lambda_{n'+2}(N^{n'-n}R_{n'-n})^{-1}\lambda_n^{-1}$, which is positive and finite by Lemma \ref{lemmaA1}. Hence the right hand side of (\ref{eqn31}) holds.
\end{proof}

\section*{Acknowledgments}
We are grateful to Professor Robert S. Strichartz for his continued support and encouragement for us to work on this problem. We would also like to thank Professor T. Kumagai, Professor J. Kigami and Professor N. Kajino for their valuable comments and pointing out some gaps on a previous version of this manuscript.

\subsection*{Conflicts of interest} The Authors declare that there is no conflict of interest.

\subsection*{Data availability statement.} Data sharing not applicable to this article as no datasets were generated or analysed during the current study.

\bibliographystyle{amsplain}

\end{document}